\documentclass[11pt, a4paper]{article}

\usepackage{inputenc} 
\usepackage[english]{babel}
\usepackage{dsfont}
\usepackage{enumerate}

\usepackage[totalheight=24 true cm, totalwidth=17 true cm]{geometry}
\usepackage{enumitem}

\usepackage{bbding}
\usepackage{amsmath,multirow}
\usepackage{amsthm}
\usepackage{amssymb}
\usepackage{subcaption}
\usepackage{mathrsfs}
\usepackage{mathtools}
\usepackage{ stmaryrd }
\usepackage{url}
\usepackage{wrapfig}
\usepackage{starfont}
\usepackage{pifont}
\usepackage{eurosym}
\usepackage{ulem}
\usepackage{cancel}
\usepackage{array}
\usepackage{booktabs}
\usepackage{tikz}
\usepackage{tikz}
\usetikzlibrary{automata, positioning, arrows}
\usetikzlibrary{calc}

\usepackage{imakeidx}
\makeindex[]

\usepackage{multicol}
\usepackage{relsize}
\usepackage{float}
\usepackage{tikz,pgf}
\usepackage{tikz-cd}
\usepackage{wasysym}
\usepackage{graphicx}
\usepackage{fancyhdr}
\usepackage{verbatim}

\usepackage{pgfplots}
\pgfplotsset{compat=1.15}
\usepackage{mathrsfs}
\usetikzlibrary{arrows}
\usepackage{parskip}
\usepackage{tocloft}
\usetikzlibrary{positioning}

\usepackage[all,dvips,knot,web,arc,curve,color,frame]{xy}
\usepackage[all,knot,arc,color,web]{xy}
\xyoption{arc}
\xyoption{web}
\xyoption{curve}

\numberwithin{equation}{section}

\theoremstyle{plain}
\newtheorem{theorem}{Theorem}[section]
\newtheorem{proposition}[theorem]{Proposition}
\newtheorem{corollary}[theorem]{Corollary}
\newtheorem{lemma}[theorem]{Lemma}

\theoremstyle{definition}
\newtheorem{definition}[theorem]{Definition}

\newtheorem{remark}[theorem]{Remark}

\newtheorem{question}{Question}
\usepackage{amsthm}

\newtheorem{example}{Example}

\AtEndEnvironment{example}{\hfill$\triangle$}

\usepackage[bookmarksnumbered=true]{hyperref} 
\hypersetup{
     colorlinks = true,
     linkcolor = blue,
     anchorcolor = blue,
     citecolor = teal,
     filecolor = blue,
     urlcolor = blue
     }

\everymath{\displaystyle}
\allowdisplaybreaks

\makeatletter
\def\mathcenterto#1#2{\mathclap{\phantom{#1}\mathclap{#2}}\phantom{#1}}
\let\old@widetilde\widetilde
\def\widetildeto#1#2{\mathcenterto{#2}{\old@widetilde{\mathcenterto{#1}{#2\,}}}}
\let\old@widehat\widehat
\def\widehatto#1#2{\mathcenterto{#2}{\old@widehat{\mathcenterto{#1}{#2\,}}}}
\makeatother

\renewcommand{\emph}[1]{\textit{#1}}

\hypersetup{
    colorlinks=true,
    linkcolor=blue,
    filecolor=magenta,      
    urlcolor=cyan,
    pdftitle={Overleaf Example},
    pdfpagemode=FullScreen,
    }
    
\setlist[itemize]{noitemsep}

\normalfont

\usepackage{bbm}
\usepackage{dsfont}

\usepackage{url}
\usepackage{amsmath,blkarray,booktabs, bigstrut}
\usepackage{tikz}
\usepackage{multirow}
\usepackage{xcolor}
\usepackage{color}
\newtheorem*{question*}{Question}

\title{Discrete-time, discrete-state  multistate Markov models from the perspective of algebraic statistics}
\usepackage{xcolor}
\author{Dario Gasbarra, Kaie Kubjas, Sangita Kulathinal,  \\ Nataliia Kushnerchuk, Fatemeh Mohammadi, Etienne Sebag}
\date{\today}

\begin{document}

\maketitle
\begin{abstract}
\noindent We study discrete-time, discrete-state multistate Markov models from the perspective of algebraic statistics. These models are commonly studied in event history analysis, and are characterized by the state space, the initial distribution and the transition probabilities. The event histories provide information on the states occupied at instances over a time period and the transitions between states are governed by the multistate model. A finite path  under the multistate Markov model is a particular set of states occupied at finite time instances $\{1, \dots, n\}$. The main goal of this paper is to establish a bridge between event history analysis and algebraic statistics. The joint probabilities of finite paths in these models has a natural monomial parametrization in terms
of initial distribution and transition probabilities. We study the polynomial relations among joint path probabilities. We show that nonhomogeneous multistate Markov models of arbitrary order when the statistical constraints on the parameters are disregarded, are slices of decomposable hierarchical models. This yields a complete description of their vanishing ideals as toric ideals generated by explicit families of binomials. Moreover, the variety of this vanishing ideal equals the nonhomogeneous multistate Markov model on the probability simplex. In contrast, homogeneous multistate Markov models exhibit different algebraic behavior, as time homogeneity imposes additional polynomial relations, leading to vanishing ideals that are strictly larger than in the nonhomogeneous case. We derive families of binomial relations that vanish on homogeneous multistate Markov models and show with examples that
these relations form only a partial generating set of the vanishing ideal in general. Finally, we analyze maximum likelihood estimation from both statistical and algebraic perspectives. For nonhomogeneous multistate Markov models, the classical formula in statistics for computing maximum likelihood estimates agrees with the formula for nondecomposable hierarchical models. In the homogeneous setting, examples show that solving the MLE problem algebraically can be substantially more involved than using classical statistical methods in this setting. Finally, we provide real data applications where we apply the statistical theory to obtain the maximum likelihood estimates of the parameters under specific multistate Markov models.
\end{abstract}
{
\hypersetup{
  linkcolor=black,
  citecolor=black,
  urlcolor=black
}
\tableofcontents
}
\section{Introduction}\label{sec:Intro}
Event history analysis provides a statistical framework for modeling processes that evolve over time
through sequences of events. A central class of models in this area is given by multistate Markov
models, which describe such processes via transitions between a finite set of states. In discrete
time, the probability of an observed path or a trajectory factors multiplicatively in terms of transition
probabilities, giving rise to polynomial relations among joint path probabilities. In this paper, we
study the algebraic structure underlying these models.

\subsection{Event history and multistate Markov models}
We begin by recalling the statistical formulation of multistate Markov models and by presenting several canonical examples drawn from event history analysis. These examples serve two purposes: they illustrate the modeling assumptions underlying multistate Markov models, and they highlight the common structural feature that will be exploited throughout the paper, namely that observations correspond to finite paths whose probabilities factor as products of the initial distribution and transition probabilities.

In this context, the term \emph{event history} refers to information and data about events occurring during a unit’s lifetime.
Units and events usually carry context-dependent definitions. For example, in life history studies, units are people (or subjects, individuals), and events are outcomes related to health, such as the diagnosis of a specific disease, birth, death, fertility, or sociodemographic status (education, employment, marriage). In engineering applications, units are systems, and events are related to the functioning of the system, such as transitions between normal, defective, and failed states.

Event history analysis is a multidisciplinary area of statistics concerned with the timing and occurrence of one or multiple events during a unit’s lifetime. Such analyses treat event times as random variables and account for both their occurrence and temporal ordering, with modeling choices depending on the study design and data-generating mechanism. A common feature of event history data is that events are recorded over time for each unit, leading to longitudinal data structures. Multistate models provide a flexible framework for analyzing such processes and are widely used across the life sciences, engineering, insurance, economics, and social sciences.

A multistate model characterizes an event history process by specifying a discrete state space $S$ together with the allowed transitions between states. Event history data are used to estimate the transition rates 
under the assumed multistate model. When an event history process is modeled using a multistate framework, each event corresponds to a specific type of transition between two states. Both continuous- and discrete-time multistate models are studied in the literature. We will consider discrete-time multistate models under Markov assumption (Section~\ref{sec:Preliminaries}) in this paper and refer to them as multistate Markov models or multistate Markov chains. In this paper, we use the term Markov chains to indicate Markov chains with all possible transitions. With this convention, it is clear that a multistate Markov chain will typically have forbidden transitions according to the structure of the multistate model. For more information on  discrete-time multistate Markov chains, we refer readers to \cite[Appendix A.2]{Aalen2008} and  \cite[Section 1.3]{Cook2018}, and for applications, to \cite{BORGAN2006, Diggle2007, Dudel2020, Suresh2022} and references therein.
We present three examples to illustrate this idea below. We will return to the illness-death model and its generalization in later sections to illustrate the algebraic description, ideals and related topics.

\textbf{Simple survival model.} We begin with the simple survival model, in which an individual is assumed to be alive at the start of the study and is followed until death.
In this setting, only one event (death) is possible. This event history process can be modeled using the multistate model shown in Figure~\ref{fig:survival}. The state space is $S=\{0, 1\}$, with initial state $0$, and the event corresponds to the transition from state $0$ to state $1$. This model satisfies the Markov assumption, since the current state depends only on the most recent past state and not on the entire history of the process.  The labels \emph{alive} and \emph{dead} are generic and should be interpreted according to the application.
For example, in an engineering context, the survival model may describe the transition of a system from a functioning state to a failed state.

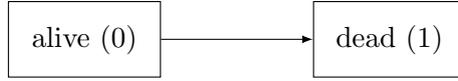
\begin{figure}[!ht]
\begin{center}
\begin{tikzpicture}[
    sharp corners=2pt,
    inner sep=5pt,
    node distance=2cm,
    >=latex]
\tikzstyle{my node}=[draw,minimum height=1cm,minimum width=2cm]
\node[my node] (alive){alive $(0)$};
\node[my node,right=of alive](dead){dead $(1)$};
\draw[->] (alive) -- (dead);
\end{tikzpicture}
\end{center}
\caption{Simple survival model. An individual starts in state $0$ and remains in that state until transitioning to state $1$. For three discrete time points, the process admits three possible paths (or trajectories): $(0,0,0)$, $(0,0,1)$, and $(0,1,1)$. Once the process enters state $1$, it remains there; in this example, state $1$ is an absorbing state.}\label{fig:survival}
\end{figure}
\textbf{2-state model.} This model extends the simple survival model by allowing transitions from state $1$ back to state $0$ (Figure~\ref{fig:2state}). Applications include transitions between susceptible and infected states, the dynamics of disease activity or biomarkers, and the functioning of a system. The state space is $S=\{0,1\}$, and two types of events are possible: transitions from state $0$ to state $1$ and from state $1$ to state $0$. A \emph{no-transition} is typically not regarded as an event, since the state does not change. Under this model, a unit may start in either state $0$ or $1$.

\begin{figure}[!ht]
\begin{center}
\begin{tikzpicture}[
    sharp corners=2pt,
    inner sep=5pt,
    node distance=2cm,
    >=latex]
\tikzstyle{my node}=[draw,minimum height=1cm,minimum width=2cm]
\node[my node] (healthy){healthy $(0)$};
\node[my node,right=of healthy](diseased){diseased $(1)$};
\draw[->] (healthy.5) -- (diseased.175);
\draw[->] (diseased) -- (healthy);
\end{tikzpicture}
\end{center}
\caption{An alternating 2-state process. An individual starts in state $0$ and remains there until transitioning to state $1$. From state $1$, the process may return to state $0$, and multiple transitions between the two states are possible. For three discrete time points, the process admits four possible paths (or trajectories): $(0,0,0)$, $(0,0,1)$, $(0,1,0)$, and $(0,1,1)$.}
\label{fig:2state}
\end{figure}
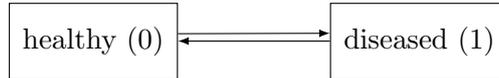

\textbf{Illness–death model.} The illness–death model is widely studied due to its many applications. For example, state $0$ may represent a disease-free state following treatment (complete remission of cancer), state $1$ may represent recurrence of the disease, and state $2$ represents death.

\begin{figure}[!ht]
    \centering
\begin{tikzpicture}[
    sharp corners=2pt,
    inner sep=5pt,
    node distance=2cm,
    >=latex]
\tikzstyle{my node}=[draw,minimum height=1cm,minimum width=2cm]
\node[my node] (healthy){healthy $(0)$};
\node[my node,right=of healthy](diseased){diseased $(1)$};
\node[my node] at ($(healthy)!0.5!(diseased)-(0pt,1.5cm)$) (dead) {dead $(2)$};
\draw[->] (healthy) -- (dead);
\draw[->] (diseased) -- (dead);
\draw[->] (healthy) -- (diseased);
\end{tikzpicture}
   \caption{An illness-death model. The initial state is $0$, and state $2$ is absorbing. For three discrete~time points, the process admits six possible paths: 
   $(0,0,0)$, $(0,0,1)$, $(0,1,1)$, $(0,0,2)$, $(0,1,2)$,~and~$(0,2,2)$.}
    \label{fig:illness}
\end{figure}
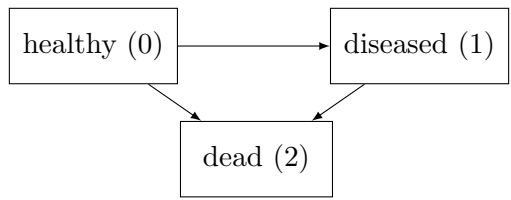

Multistate models are fully characterized by the state space, the initial distribution and the transition probabilities. In this paper, we focus on discrete-time, discrete-state multistate Markov chains and describe the dynamics of the chains via transition probabilities. Under these assumptions, the probability of an observed path admits a multiplicative factorization determined by the initial distribution and successive state transition probabilities (see Section~\ref{sec:Preliminaries}).
This factorization induces a natural monomial parametrization of joint path probabilities and thereby leads to natural questions from the perspective of algebraic statistics. While the first-order Markov chains have been studied in algebraic statistics (see, e.g., \cite[Section 1.1]{sullivant2023algebraic}), the $k$-th order multistate Markov models, where only a subset of transitions are allowed and the current state depends on the past $k$ states, are not formally transcribed, to our knowledge, in the language of algebraic statistics. In addition, a higher-order Markov chain may not be treated as a first-order Markov chain because not all properties of first-order Markov chains carry over to higher-order Markov chains \cite{Xu2026}, and the purposes for which they are used may also differ.

\subsection{An algebraic perspective on multistate Markov models}

In all these examples, an observation corresponds to a finite path through the state space, and its
probability admits a multiplicative factorization determined by successive state transitions. This
factorization depends only on the Markov structure, the order of dependence, and possible structural
constraints such as forbidden transitions or absorbing states. From an algebraic perspective, this multiplicative structure induces a monomial parametrization. 
The Zariski closure of the image of this monomial parametrization is an algebraic variety that contains the multistate Markov model as a subset.

When imposing the statistical constraints on the parameters of the monomial parametrization, the vanishing ideal of the above variety captures all polynomial constraints among joint path probabilities implied by the model assumptions. As a result, determining generators of this ideal provides an algebraic characterization of the model and has direct consequences for identifiability, model equivalence, and likelihood-based inference.
This algebraic formulation provides a unified framework for comparing different classes of multistate Markov models and for understanding how modeling assumptions translate into algebraic constraints. One caveat is that this algebraic characterization does not account for inequality constraints which would lead to a study in real algebraic geometry.

Algebraic methods play a
central role in discrete statistical modeling
\cite{pistone2001algebraic,drton-sturmfels-sullivant,sullivant2023algebraic}.
Several algebraic approaches to statistical models have been developed in related contexts,
particularly for log-linear, graphical, and conditional independence models.
Many classical models, including log-linear
and graphical models, are equal to a probability simplex intersected with an algebraic variety defined by binomial ideals.
Informally, the probability simplex imposes the basic probabilistic constraints that probabilities are nonnegative and sum to one, while the algebraic variety encodes the structural assumptions of the model through polynomial equations. The statistical model then consists precisely of those probability distributions satisfying both the probabilistic and algebraic constraints. See Section~\ref{sec:algebraic_description} for further details.

This perspective is particularly natural for multistate Markov models with structural restrictions. For example, a forbidden transition can be encoded by requiring the corresponding transition probability to be zero, while an absorbing state can be represented by requiring the probability of remaining in that state to be one. More generally, polynomial relations among transition probabilities can be incorporated naturally into the algebraic framework. See the discussion following Example~\ref{example:revillness_death_model} for further details.

For decomposable graphical
models, the vanishing ideals are toric when statistical constraints on parameters are disregarded and yield closed-form maximum likelihood estimators with
maximum likelihood degree one \cite{diaconis1998algebraic,dobra2003markov}. The log-linear and discrete undirected graphical models are studied from algebraic perspective in~\cite{geiger2006toric}, where decomposability is central to both polynomials vanishing on the model and rational maximum likelihood estimates.

Algebraic encoding of conditional independence beyond classical graphical models has also been
developed using combinatorial structures~\cite{herzog2010binomial}. In particular, lattice conditional independence
models relate collections of conditional independence statements to distributive lattices and Hibi
ideals, extending beyond standard DAG representations
\cite{caines2022lattice,ene2011monomial}. 
However, multistate Markov models do not fit directly into these frameworks: their structural assumptions such as time-homogeneity or forbidden transitions induce algebraic constraints that go beyond standard conditional independence models. Despite their central role in event history analysis \cite{andersen2012multi}, these algebraic features have not been systematically explored.
Within this algebraic context, the present work provides a unified algebraic-statistical study of discrete-time multistate Markov models.

The main goal of this paper is to establish a systematic connection between multistate Markov models and algebraic statistics. This translation makes it possible to study multistate models using algebraic tools such as ideals, varieties, and polynomial constraints. In particular, structural features of multistate models, such as forbidden transitions and absorbing states, have a natural algebraic description. Examples~\ref{example:illness_death_model} and~\ref{example:revillness_death_model} illustrate how these models can be translated into the language of algebraic statistics and motivate the algebraic questions studied in the remainder of the paper.

\subsection{Main results and contributions}

We summarize below the main results and contributions of this paper.
We formulate discrete-time multistate Markov models as algebraic statistical models via monomial
parametrizations of joint path probabilities
(Section~\ref{sec:algebraic_description}). In the following two sections we study the associated vanishing ideals. For the nonhomogeneous setting, we prove that $k$-th order multistate Markov models when the statistical constraints on the parameters are disregarded
are decomposable hierarchical models (Lemma~\ref{lem:markov_is_hierarchical}). This identification
allows us to use existing results on decomposable models to obtain an explicit description
of the defining equations in terms of binomial generators
(Proposition~\ref{prop:generating set general}), building on
\cite{dobra2003markov,diaconis1998algebraic}. We further show that models
with forbidden transitions or absorbing states arise as coordinate slices of the general
nonhomogeneous model, and we describe how their vanishing ideals are obtained by adding linear
relations (Proposition~\ref{prop:restricted_model_slice} and
Corollary~\ref{lem:generators_restricted}). We show that in the nonhomogeneous setting the $k$-th order multistate model is equal to its Zariski closure intersected with the probability simplex (Proposition~\ref{prop:equal_ideal_with_constraints}). 

In contrast, we show that homogeneous multistate Markov models satisfy additional algebraic
relations. 
We derive a general family of binomial
relations induced by time-homogeneity (Proposition~\ref{prop:homogeneous binomials}) and show
through explicit examples that these relations do not suffice to generate the full vanishing
ideal even when we disregard the statistical constraints on the parameters. We also analyze maximum likelihood estimation from both statistical and algebraic
perspectives. In the nonhomogeneous case, we discuss that the formula for maximum likelihood estimates for decomposable hierarchical models agrees with the classical formulas from statistics, while in the homogeneous case we see that the algebraic formulation leads to more challenges than the
classical statistical approach (Section~\ref{sec:mle}). As a consequence we derive that in the homogeneous setting the $k$-th order multistate model can be strictly included in its Zariski closure interested with the probability simplex (Proposition~\ref{prop:model_is_not_equal_to_variety_interesected_with_probability_simplex}). 
Finally, we illustrate the theory on classical multistate Markov models from event history analysis,
including illness-death type models and models with structural constraints, highlighting the
practical implications of the algebraic structure for inference and model interpretation
(Section~\ref{sec:applications}).

We emphasize that the algebraic viewpoint is used here to understand structure and constraints;
in some cases, classical statistical methods remain simpler for inference.

\smallskip
\noindent{\bf Outline.}
Section~\ref{sec:Preliminaries} introduces discrete-time, discrete-state multistate Markov models and recalls the
factorization of finite path probabilities.~Section~\ref{sec:algebraic_description} presents the
algebraic viewpoint via monomial parametrizations and the associated vanishing ideals.
Section~\ref{sec:implicitization_nonhomogeneous} develops the main structural results for the
nonhomogeneous setting, relating these models to decomposable hierarchical models and
studying their vanishing ideals. Section~\ref{sec:hom} turns to the homogeneous setting and
shows that time-homogeneity imposes additional algebraic constraints. Section~\ref{sec:mle}
studies maximum likelihood estimation and its relation to the algebraic structure of the models.
Section~\ref{sec:applications} illustrates the theory on classical multistate models from event
history analysis. Section~\ref{sec:discussion} concludes with further directions and open problems.

\medskip\noindent{\bf Code availability.}
The computational examples are publicly available at
\url{https://github.com/FatemehMohammadi/algebraic-multistate-markov-models}.
The repository contains the Macaulay2 implementations for Examples~\ref{example:illness_death_model},~\ref{example:revillness_death_model},~\ref{example:strict_inclusion_homogeneous}, and~\ref{example:homogeneous_MLE}, together with the R code and data used for the computations in Section~\ref{sec:applications}.

\section{Discrete-time, discrete-state multistate Markov models}\label{sec:Preliminaries}
We consider discrete-time, discrete-state Markov chains and briefly review the basic definitions and the factorization of path probabilities. Throughout the paper, motivated by applications in survival and event history analysis, we use the term \emph{Markov chain model} for an unrestricted discrete-time, discrete-state Markov model, and \emph{multistate Markov model} for a model with forbidden transitions.

\medskip
\noindent \textbf{First-order Markov chains.}
A discrete-time stochastic process $\{X_{\ell} : \ell \ge 1\}$ with a finite state space $S$ is a first-order  Markov chain if, for each time $\ell\ge 2$, the conditional distribution of $X_{\ell}$ given the past depends only on the most recent state $X_{\ell-1}$. That is, for all $\ell \ge 2$ and all $i,j \in S$,
\begin{align}
P(X_\ell = j \mid X_1 = i_1, \dots, X_{\ell-2} = i_{\ell-2}, X_{\ell-1} = i)
&= P(X_\ell = j \mid X_{\ell-1} = i)
:= a_{ij}^{(\ell)}. \label{eqn:MP}
\end{align}
The condition \eqref{eqn:MP} is called the Markov property. If the transition probabilities do not depend on $\ell$, the chain is \emph{time-homogeneous}, i.e.\ $a_{ij}^{(\ell)}=a_{ij}$ for all $\ell$ and all $i,j\in S$.

\medskip
\noindent The joint probability of a length-$n$ path $(i_1,\ldots,i_n)$ can be written in terms of transition~probabilities~as
\begin{align}
p_{i_1 \ldots i_n}
&= P(X_1=i_1,\ldots,X_n=i_n) \nonumber\\
&=  P(X_1=i_1) \prod_{\ell=2}^n P(X_{\ell}=i_{\ell} \mid X_1=i_1,\ldots,X_{\ell-1}=i_{\ell-1}) \nonumber\\
&= \pi_{i_1}\prod_{\ell=2}^n P(X_{\ell}=i_{\ell} \mid X_{\ell-1}=i_{\ell-1})
= \pi_{i_1}\prod_{\ell=2}^n a_{i_{\ell-1}\, i_{\ell}}^{(\ell)}, \label{eqn:pathprob}
\end{align}
where $\pi_{i_1}=P(X_1=i_1)$ denotes the initial distribution. In particular, if the initial state is fixed (e.g.\ $\pi_0=1$), then the factor $\pi_{i_1}$ can be omitted in~\eqref{eqn:pathprob}.

Since $\pi_{i_1}$ and $a_{i_{\ell-1}\, i_{\ell}}^{(\ell)}$ represent (conditional) probabilities, they satisfy 
\begin{align*}
& \biggl\{ \pi_{i_1}  \geq 0 \; \text{for all} \;  i_1 \in S, \;  \text{and} 
\sum_{i_1 \in S} \pi_{i_1}=1, \; \text{and} \\ 
& a_{i_{\ell-1}\, i_{\ell}}^{(\ell)}  \geq 0, \; i_{\ell-1},i_{\ell} \in S, \text{and} \; \sum_{i_{\ell} \in S} a_{i_{\ell-1}\, i_{\ell}}^{(\ell)}=1,
\text{for all} \; \ell \in \{2,\ldots,n\}\;  \text{and} \; 
  i_{\ell-1} \in S \biggr\}.
\end{align*}
The first-order Markov chains are well-studied in algebraic statistics, see~\cite[Section 1.1]{sullivant2023algebraic}.

\noindent \textbf{Higher-order Markov chains.}
More generally, a discrete-time stochastic process $\{X_{\ell} : \ell \ge 1\}$ with state space $S$ is a $k$-th order Markov chain if, for each $\ell > k$, the conditional distribution of $X_{\ell}$ given the past depends only on the most recent $k$ states. That is, for all $\ell > k$ and all states in $S$,
\begin{align}
P(X_{\ell} = i_{\ell} \mid X_1 = i_1, \dots, X_{\ell-1} = i_{\ell-1})
&= P(X_{\ell} = i_{\ell} \mid X_{\ell-k}=i_{\ell-k}, \dots, X_{\ell-1}=i_{\ell-1}) \nonumber\\
&:= a^{(\ell)}_{i_{\ell-k}\cdots i_{\ell-1} i_{\ell}}. \label{eqn:MP-k}
\end{align}
The condition \eqref{eqn:MP-k} is the Markov property for the $k-$th order Markov chain.  In the time-homogeneous case, $a^{(\ell)}_{i_{\ell-k}\cdots i_{\ell-1} i_{\ell}}=a_{i_{\ell-k}\cdots i_{\ell-1} i_{\ell}}$ for all $i_{\ell-k},\ldots,i_{\ell} \in S$ and  $\ell > k$.
The joint probability of a length-$n$ path $(i_1,\ldots,i_n)$, $n > k$, can be written in terms of transition~probabilities~as \\
\begin{align}
p_{i_1 \ldots i_n}
&= P(X_1=i_1,\ldots,X_n=i_n) \nonumber\\
&= P(X_1=i_1,\ldots,X_k=i_k)\prod_{\ell=k+1}^n P(X_{\ell}=i_{\ell} \mid X_1=i_1,\ldots,X_{\ell-1}=i_{\ell-1}) \nonumber\\
&= \pi_{i_1 \ldots i_k}\prod_{\ell=k+1}^n P(X_{\ell}=i_{\ell} \mid X_{\ell-k}=i_{\ell-k}\ldots X_{\ell-1}=i_{\ell-1})
= \pi_{i_1 \ldots i_k}\prod_{\ell=k+1}^n a^{(\ell)}_{i_{\ell-k}\cdots i_{\ell-1} i_{\ell}}, \label{eqn:pathprob-k}
\end{align}
where $\pi_{i_1 \ldots i_k}=P(X_1=i_1,\ldots,X_k=i_k)$ denotes the initial distribution.

Since $\pi_{i_1 \ldots i_k}$ and $a^{(\ell)}_{i_{\ell-k}\cdots i_{\ell-1} i_{\ell}}$ represent (conditional) probabilities, they satisfy 
\begin{align}
& \biggl\{ \pi_{i_1 \ldots i_k} \geq 0, \;  \text{for all} \;  i_1,\ldots,i_k \in S, \; \; \sum_{i_1,\ldots,i_k \in S} \pi_{i_1 \ldots i_k}=1, \; \text{and} \nonumber \\ 
& a^{(\ell)}_{i_{\ell-k}\cdots i_{\ell-1} i_{\ell}} \geq 0, \;  
\;  
\sum_{i_{\ell} \in S} a^{(\ell)}_{i_{\ell-k}\cdots i_{\ell-1} i_{\ell}}=1, \; \text{for all} \;  \ell \in \{k+1,\ldots,n\}, i_{\ell-k},\cdots,i_{\ell-1}\in S \biggr\}. \label{constraints:korder}
\end{align}
\noindent\textbf{Models with forbidden transitions.}
In practice, the nature of the underlying process may impose additional constraints on the admissible
transitions between states in the process $\{X_{\ell}, \ell \ge 1\}$. In such cases, the resulting model
can be viewed as a slice of the general model. The assumption that certain transitions are forbidden
is typically expressed by setting the corresponding transition probabilities to zero, which in turn
implies that all paths containing forbidden transitions have probability zero. 

For example, under the illness-death model (Figure~\ref{fig:illness}), transitions from state
$1 \rightarrow 0$ are forbidden, so $P(X_{\ell} = 0 \mid X_{\ell - 1} = 1) = a_{10} = 0$ for all
$\ell > 1$. In other applications, some states may be absorbing, such as the state ``dead'' in
Figure~\ref{fig:illness}, in which case $P(X_{\ell} = j \mid X_{\ell - 1} = 2) = a_{2j}=0$ for
$j=0,1$, and $P(X_{\ell} = 2 \mid X_{\ell - 1} = 2) = a_{22} = 1$. By the same token, the initial
state may also be restricted, depending on the research question and study setting.
Under a Markov chain model with three states, the number of possible length-$4$ paths
$(X_1,\dots,X_4)$ is $3^4$, whereas under the illness-death model (a three-state model with forbidden
transitions), with no restriction on the initial state, there are only $14$ possible paths
(Table~\ref{tab:obsevations4}).

For another illustrative example of forbidden transitions, we consider a Markov chain on the alphabet
$\texttt{a-z}$ together with a space symbol, which represents an absorbing state. Each word then
corresponds to a trajectory of the chain. We restrict attention to \emph{alternating words}, in which
transitions from vowel to vowel and from consonant to consonant are forbidden, as in the word
\emph{honorificabilitudinitatibus}.

The word \textit{honorificabilitudinitatibus} was an uncommon one we found in the data, and the structure it follows is quite restrictive  because English spelling is not generally governed by alternating words, in which transitions of the form vowel to consonant to vowel to consonant and so on are the only possibilities. Moreover, the transitions are not governed only by the previous letter. The likelihood of a subsequent letter depends on characteristics such as whether or not we are at the beginning or end of a word, whether the current segment is a prefix or a suffix, whether we are inside a consonant cluster, whether the word has Latin roots, for example, and so on. If we still wanted to use this model structure, then a hidden Markov model could be appropriate to create latent states that precisely capture the linguistic structure from the observed letters. In Section~\ref{sec:applications}, we will actually analyze a version of these data that do not comprise the word \emph{honorificabilitudinitatibus}, in order to illustrate maximum likelihood estimation under two multistate Markov models from both statistical and algebraic perspectives.

 \section{Algebraic statistical framework}\label{sec:algebraic_description}
In this section we translate the probabilistic factorization of path probabilities from
Section~\ref{sec:Preliminaries} into an algebraic  framework. The key point is that the
Markov property yields a monomial parametrization of the joint path probabilities in terms of an
initial distribution and transition parameters. We formalize this parametrization, and explain how time-homogeneity and forbidden transitions fit into this
setup. Finally, we recall basic facts on toric ideals, since the models considered here are defined
by monomial parametrizations and hence have toric vanishing ideals.

 \subsection{Parametrization and implicitization}
\label{subsec:paramtrization}

The goal of this section is to describe the algebraic setup corresponding to the statistical identity~\eqref{eqn:pathprob-k}. The factorization of joint path probabilities in \eqref{eqn:pathprob-k} defines a monomial parametrization of the joint distribution on paths.
We define the parametrization map
\[
\phi:\;
\mathbb{R}^{S^k} \times \prod_{\ell=k+1}^n \mathbb{R}^{S^{k+1}}
\longrightarrow \mathbb{R}^{S^n},
\qquad
\phi_{i_1\ldots i_n}\!\left(
\pi_{i_1\ldots i_k},
\bigl(a^{(\ell)}_{i_{\ell-k}\cdots i_{\ell}}\bigr)_{\ell=k+1}^n
\right)
=
\pi_{i_1\ldots i_k}
\prod_{\ell=k+1}^n a^{(\ell)}_{i_{\ell-k}\cdots i_{\ell-1} i_\ell}.
\]
The pullback of $\phi$ is the ring homomorphism
\[
\phi^*: \mathbb{R}[p_{i_1\ldots i_n}] \rightarrow \mathbb{R}\bigl[\,
\pi_{i_1 \ldots i_k},\;
a^{(\ell)}_{i_{\ell}\, i_{\ell+1}\cdots i_{\ell+k}}
\;\big|\;
\ell=k+1,\ldots,n
\bigr], \quad \phi^*(p_{i_1 \ldots i_n})= \pi_{i_1 \ldots i_k}\prod_{\ell=k+1}^n a^{(\ell)}_{i_{\ell-k}\cdots i_{\ell-1} i_{\ell}}.
\]
In the homogeneous setting, the transition parameters
$a^{(\ell)}_{i_{\ell-k} i_{\ell-k+1} \cdots i_{\ell}}$
coincide for all $\ell$. We therefore drop the upper index and write
$a_{i_{\ell-k} i_{\ell-k+1} \cdots i_{\ell}}$ in the above formulas. We use the same notation for both parametrizations as it is usually clear from the context whether we are in nonhomogeneous or homogeneous setting.

Forbidden transitions are incorporated by setting the corresponding transition parameters to zero, yielding coordinate slices of the general parametrization.

We study the kernel of the pullback $\phi^*$, denoted $I:=\ker(\phi^*)$, which is an ideal in
$\mathbb{R}[p_{i_1\ldots i_n}]$. The elements of the kernel are polynomials in the joint probabilities $p_{i_1\ldots i_n}$. These polynomials vanish on all the joint probabilities in the corresponding models.  They also vanish at other points that may not correspond to joint probability distributions at all. Determining generators of this ideal is known as the implicitization
problem and provides a complete algebraic description of the model. In the next two sections, we
investigate this problem in the nonhomogeneous and homogeneous settings. For more information on the implicitization problem, we refer the reader to \cite[Chapter 3 
\S3]{cox2007ideals}.

\begin{example}[Illness-death model] \label{example:illness_death_model}
We assume that the illness-death model (Figure~\ref{fig:illness}) is a first-order multistate Markov model with three states.  This example corresponds to forbidden transitions according to Figure~\ref{fig:illness} where the initial state can be $0$ or $1$. The choice of the initial state typically depends on the research question. Table~\ref{tab:obsevations4} gives the path probabilities for four time instances under this model.
This example illustrates that the implicitization ideals in the homogeneous and nonhomogeneous settings differ substantially in structure and size.

Homogeneous ideal for Table~\ref{tab:obsevations4}:
\begin{equation*}
\begin{gathered}
 \langle p_{1 1 2 2}^2-p_{1 1 1 2}p_{1 2 2 2},\ p_{0 1 2 2}p_{1 1 2 2}-p_{0 1 1 2}p_{1 2 2 2},\ p_{0 0 2 2}p_{0 1 2 2}-p_{0 0 1 2}p_{0 2 2 2},\\
  p_{1 1 1 2}p_{0 1 2 2}-p_{01 1 2}p_{1 1 2 2},\ p_{1 1 1 1}p_{0 1 2 2}-p_{0 1 1 1}p_{1 1 2 2},\ p_{0 0 2 2}^2-p_{0 0 0 2}p_{0 2 2 2},\\
  p_{0 0 1 2}p_{0 0 2 2}-p_{0 0 0 2}p_{0 1 2 2},\ p_{1 1 11}p_{0 1 1 2}-p_{0 1 1 1}p_{1 1 1 2},\ p_{1 1 1 1}p_{0 0 1 2}-p_{0 0 1 1}p_{1 1 1 2},\\
  p_{0 1 1 1}p_{0 0 1 2}-p_{0 0 1 1}p_{0 1 1 2},\ p_{0 0 1 1}p_{0 0 1 2}-p_{00 0 1}p_{0 1 1 2},\ p_{0 0 1 1}^2-p_{0 0 0 1}p_{0 1 1 1},\\
  p_{0 0 1 2}p_{1 1 2 2}p_{0 2 2 2}-p_{0 1 1 2}p_{0 0 2 2}p_{1 2 2 2},\ p_{0 0 1 1}p_{1 1 2 2}p_{0 2 22}-p_{0 1 1 1}p_{0 0 2 2}p_{1 2 2 2},\\
  p_{0 0 0 1}p_{1 1 2 2}p_{0 2 2 2}-p_{0 0 1 1}p_{0 0 2 2}p_{1 2 2 2},\ p_{0 0 0 1}p_{1 1 1 2}p_{0 2 2 2}-p_{0 1 1 1}p_{0 00 2}p_{1 2 2 2},\\
  p_{0 1 1 2}p_{0 0 2 2}p_{1 1 2 2}-p_{0 0 1 2}p_{1 1 1 2}p_{0 2 2 2},\ p_{0 1 1 1}p_{0 0 2 2}p_{1 1 2 2}-p_{0 0 1 1}p_{1 1 1 2}p_{0 2 2 2},\\
  p_{00 1 1}p_{0 0 2 2}p_{1 1 2 2}-p_{0 1 1 1}p_{0 0 0 2}p_{1 2 2 2},\ p_{0 0 0 1}p_{0 0 2 2}p_{1 1 2 2}-p_{0 0 1 1}p_{0 0 0 2}p_{1 2 2 2},\\
  p_{0 0 0 2}p_{0 1 22}^2-p_{0 0 1 2}^2p_{0 2 2 2},\ p_{0 0 1 1}p_{1 1 1 2}p_{0 0 2 2}-p_{0 1 1 1}p_{0 0 0 2}p_{1 1 2 2},\\
  p_{0 0 0 1}p_{1 1 1 2}p_{0 0 2 2}-p_{0 0 1 1}p_{0 0 0 2}p_{11 2 2},\ p_{0 0 1 1}p_{0 1 1 2}p_{0 0 2 2}-p_{0 1 1 1}p_{0 0 0 2}p_{0 1 2 2},\\
  p_{0 0 0 1}p_{0 1 1 2}p_{0 0 2 2}-p_{0 0 1 1}p_{0 0 0 2}p_{0 1 2 2},\ p_{0 0 12}^2p_{1 1 1 2}p_{0 2 2 2}-p_{0 0 0 2}p_{0 1 1 2}^2p_{1 2 2 2}\rangle.
\end{gathered}
\qedhere\end{equation*}

Nonhomogeneous ideal for Table~\ref{tab:obsevations4}:
\begin{gather*}
    \langle
p_{1 1 1 2}p_{0 1 2 2}-p_{0 1 1 2}p_{1 1 2 2},\ p_{1 1 1 1}p_{0 1 2 2}-p_{0 1 1 1}p_{1 1 2 2},\ p_{1 1 1 1}p_{0 1 1 2}-p_{0 1 1 1}p_{1 1 1 2},\\
    p_{1 1 1 1}p_{0 01 2}-p_{0 0 1 1}p_{1 1 1 2},\ p_{0 1 1 1}p_{0 0 1 2}-p_{0 0 1 1}p_{0 1 1 2} 
\rangle.
\end{gather*}

{\renewcommand{\arraystretch}{1.2}%
\begin{table}[!ht]
    \centering 
    \caption{Potential observations at four time instances under the illness-death model (Figure~\ref{fig:illness}) assumed to be a first-order multistate Markov model. The initial state can be $0$ or $1$. Transition from state $1$ to $0$ is forbidden so $a_{10}^{(\ell)} = 0,$ for all $\ell$. State $2$ is an absorbing state and hence, the chain will stay in state $2$ once it is  reached. The transition probabilities from state $2$ to states $0$ and $1$ are all zero, and to state $2$ is one. Hence, all probabilities of the type $a_{20}^{(\ell)} = 0, a_{21}^{(\ell)} = 0, a_{22}^{(\ell)} = 1,$ for all $\ell$.}\label{tab:obsevations4}
    \begin{tabular}{c|c|c|r}
    \hline
        Initial & Three next instances & Prob (nonhomogeneous) & Prob (homogeneous)\\
        \hline
        0 & 000 & $\pi_0a_{00}^{(1)} a_{00}^{(2)} a_{00}^{(3)}$ & $\pi_0a_{00}^3$ \\
        0 & 001  & $\pi_0a_{00}^{(1)} a_{00}^{(2)} a_{01}^{(3)}$ & $\pi_0a_{00}^2 a_{01}$ \\
        0 & 002  & $\pi_0a_{00}^{(1)} a_{00}^{(2)} a_{02}^{(3)}$ & $\pi_0a_{00}^2 a_{02}$ \\
        0 & 011  & $\pi_0a_{00}^{(1)} a_{01}^{(2)} a_{11}^{(3)}$ & $\pi_0a_{00} a_{01} a_{11}$ \\
        0 & 012  & $\pi_0a_{00}^{(1)} a_{01}^{(2)} a_{12}^{(3)}$ & $\pi_0a_{00} a_{01} a_{12}$ \\
        0 & 022  & $\pi_0a_{00}^{(1)} a_{02}^{(2)} a_{22}^{(3)}$ & $\pi_0a_{00} a_{02} a_{22}$ \\
        0 & 111  & $\pi_0a_{01}^{(1)} a_{11}^{(2)} a_{11}^{(3)}$ & $\pi_0a_{01} a_{11}^2$ \\
        0 & 112  & $\pi_0a_{01}^{(1)} a_{11}^{(2)} a_{12}^{(3)}$ & $\pi_0a_{01} a_{11} a_{12}$ \\
        0 & 122  & $\pi_0a_{01}^{(1)} a_{12}^{(2)} a_{22}^{(3)}$ & $\pi_0a_{01} a_{12} a_{22}$\\
        0 & 222  & $\pi_0a_{02}^{(1)} a_{22}^{(2)} a_{22}^{(3)}$ & $\pi_0a_{02} a_{22}^2$\\      
        \hline
        1 & 111 & $\pi_1a_{11}^{(1)} a_{11}^{(2)} a_{11}^{(3)}$ & $\pi_1a_{11}^3$ \\
        1 & 112 & $\pi_1a_{11}^{(1)} a_{11}^{(2)} a_{12}^{(3)}$ & $\pi_1a_{11}^2 a_{12}$ \\
        1 & 122  & $\pi_1a_{11}^{(1)} a_{12}^{(2)} a_{22}^{(3)}$ & $\pi_1a_{11} a_{12} a_{22}$ \\
        1 & 222  & $\pi_1a_{12}^{(1)} a_{22}^{(2)} a_{22}^{(3)}$ & $\pi_1a_{12} a_{22}^2$ \\
        \hline        
    \end{tabular}
\end{table}
}
\end{example}
\begin{example}[{Reversible illness-death model}]
\label{example:revillness_death_model}
We consider a reversible illness-death model, obtained as an extension of the illness-death model
(Figure~\ref{fig:illness} and Example~\ref{example:illness_death_model}) in which the reverse
transition from state $1$ to state $0$ is permitted. This corresponds to the transition structure in
Figure~\ref{fig:illness} with the arrow from state $1$ to $0$ included, and the initial state may be
either $0$ or $1$.
Assuming a first-order multistate Markov model, the path probabilities for four time instances are
shown in Table~\ref{tab:Revobsevations4} (Appendix~\ref{sec:appendix}). The corresponding
nonhomogeneous implicitization ideal is generated by $66$ polynomials, which are listed in
Appendix~\ref{sec:appendix}. In contrast, the implicitization ideal in the homogeneous case is
generated by $274$ polynomials.
\end{example}

Let $\Theta$ be the subset of $\textstyle{\mathbb{R}^{S^k} \times \prod_{\ell=k+1}^n \mathbb{R}^{S^{k+1}}}$ 
satisfying the statistical constraints on the parameters~\eqref{constraints:korder}.  We write $\mathcal{M}:= \phi \left( \Theta \right)$ and $\mathcal{V}:=\phi \left(\textstyle{\mathbb{R}^{S^k} \times \prod_{\ell=k+1}^n \mathbb{R}^{S^{k+1}}}
\right)$. Similarly, we write $\mathcal{M}_F$ and $\mathcal{V}_F$ for the corresponding images when we impose further constraints on the parameters in the form of forbidden transitions.  The vanishing ideal $I(\mathcal{V})$ is equal to $\ker(\phi^*)$.

For $r \in \mathbb{N}$, we define the probability simplex
\[
\Delta_r=\left\{ p \in \mathbb{R}^{r+1} | \sum_{i=1}^{r+1} p_i = 1, p_i \geq 0 \text{ for all } i \right\}.
\]
Then we have the following inclusions:
\[
\mathcal{M} \subseteq \mathcal{V} \cap \Delta_{|S|^n-1} \quad \text{and} \quad \mathcal{M}_F \subseteq \mathcal{V}_F \cap \Delta_{|S|^n-1}.
\]
We will show in Proposition~\ref{prop:equal_ideal_with_constraints} that the first inclusion is equality in the nonhomogeneous setting. We will see in Proposition~\ref{prop:model_is_not_equal_to_variety_interesected_with_probability_simplex}  that the first inclusion is in general a strict inclusion in the homogeneous setting.

The first inclusion implies $\mathcal{M} \subseteq \mathcal{V}$ and $\mathcal{M} \subseteq \Delta_{|S|^n-1}$. Hence $I(\mathcal{M}) \supseteq I(\mathcal{V})$ and $I(\mathcal{M}) \supseteq I(\Delta_{|S|^n-1})$, and therefore $I(\mathcal{M}) \supseteq I(\mathcal{V}) + I(\Delta_{|S|^n-1})$. Moreover, since $I(\mathcal{M})$ is a radical ideal, then 
\[
 \sqrt{I(\mathcal{V}) + I(\Delta_{|S|^n-1})} \subseteq I(\mathcal{M}) .
\]
Substituting $I(\mathcal{V})= \ker (\phi^*)$, $I(\Delta_{|S|^n-1}) = \langle 1-\sum_{i_1,\ldots, i_n}p_{i_1\ldots i_n}\rangle$ and $I(\mathcal{M}) = I(\phi(\Theta))$ gives
\begin{equation} \label{eqn:inclusion_of_inequalities}
\sqrt{\ker (\phi^*)+\langle 1-\sum_{i_1,\ldots, i_n}p_{i_1\ldots i_n}}\rangle\subseteq I(\phi(\Theta)).
\end{equation}
 In the homogeneous setting the latter inclusion can be strict as shown in Example~\ref{example:strict_inclusion_homogeneous}.

\subsection{Toric ideals} \label{sec:toric_ideals}

Toric ideals are classical objects in algebraic geometry, arising as defining ideals of affine toric
varieties, and have been extensively studied from geometric, algebraic, and combinatorial
perspectives. In algebraic statistics, toric ideals arise naturally as vanishing ideals of statistical
models defined by monomial parametrizations, known as log-linear models or discrete exponential
families
\cite{diaconis1998algebraic,dobra2003markov,sullivant2023algebraic,kateri2015family}.

We briefly recall the relevant definitions; see~\cite{sullivant2023algebraic} for further details.

\begin{definition}
\label{def:toric}
Let
$\phi^* : \mathbb{R}[p_1,\ldots,p_m] \longrightarrow \mathbb{R}[\theta_1,\ldots,\theta_r]$
be a ring homomorphism such that each $\phi^*(p_i)$ is a monomial in the variables
$\theta_1,\ldots,\theta_r$. The \emph{toric ideal} associated to $\phi^*$ is the kernel
\[
I_{\phi^*} := \ker(\phi^*) \subseteq \mathbb{R}[p_1,\ldots,p_m].
\]
\end{definition}

The ideal $I_{\phi^*}$ is prime and defines the Zariski closure of the image of the corresponding
monomial parametrization. Moreover, $I_{\phi^*}$ encodes all polynomial relations among the
variables $p_1,\ldots,p_m$ implied by the parametrization. For a vector
$u=(u_1,\ldots,u_m)\in\mathbb{Z}_{\ge 0}^m$, we write $p^u := p_1^{u_1}\cdots p_m^{u_m}$.
A fundamental structural result shows that $I_{\phi^*}$ is generated by binomials: for any
$u,v\in\mathbb{N}^m$ satisfying $\phi^*(p^u)=\phi^*(p^v)$, the binomial $p^u-p^v$ lies in
$I_{\phi^*}$, and such binomials generate the ideal; see, for example,
\cite[Proposition~6.2.4]{sullivant2023algebraic}.

\begin{remark}\label{rem:matrix_rep_toric}
An equivalent description of toric ideals can be given in terms of integer matrices.
Let $A\in\mathbb{Z}^{r\times m}$ be the matrix whose $i$-th column records the exponent vector of
the monomial $\phi^*(p_i)$ in the variables $\theta_1,\ldots,\theta_r$. Then the toric ideal
$I_{\phi^*}$ is generated by all binomials $p^u-p^v$ such that $Au=Av$, or equivalently,
$u-v\in\ker_{\mathbb{Z}}(A)$. In this formulation, the algebraic relations among the joint
probabilities correspond to integer dependencies among the columns of $A$. This matrix representation is particularly useful for maximum likelihood estimation on log-linear models.
\end{remark}

When the statistical constraints on the parameters are disregarded, both the homogeneous and nonhomogeneous multistate Markov models considered in this paper are defined by monomial parametrizations of joint path probabilities in terms of transition parameters.
Consequently, the corresponding vanishing ideals are toric ideals in the sense of
Definition~\ref{def:toric}. 

\section{Nonhomogeneous multistate Markov models} \label{sec:implicitization_nonhomogeneous}

In this section, we study the implicitization of nonhomogeneous $k$-th order multistate Markov models.
Our main goal is to describe the generators of the vanishing ideal of the model and to understand how
this description changes in the presence of forbidden transitions.
The key observation is that nonhomogeneous multistate Markov models when the statistical constraints on the parameters are disregarded can be identified with slices of
decomposable hierarchical models.
This identification allows us to apply existing results on the structure of vanishing ideals
of hierarchical models.
Our contribution is twofold: first, we make this connection explicit for higher-order Markov chains,
and second, we show how forbidden transitions are reflected algebraically by passing to coordinate
slices of the corresponding toric variety.
In Section~\ref{sec:hierarchical_models}, we establish the connection to hierarchical models.
In Section~\ref{sec:implicit_description_nonhomogeneous_subsection}, we describe the generators of the
vanishing ideal for the general nonhomogeneous model and its restricted variants when the statistical constraints on the parameters are disregarded.
Finally, in Section~\ref{sec:summing_to_one_condition}, we show that imposing the summing-to-one
constraint on the toric variety recovers the statistical model inside the probability simplex.

\subsection{Connection to hierarchical models}\label{sec:hierarchical_models}
In the nonhomogeneous setting, Markov chains can be viewed as decomposable hierarchical models in the sense of algebraic statistics. This observation allows one to use established results on the structure and generating sets of the associated vanishing ideals for decomposable models \cite{dobra2003markov,diaconis1998algebraic,sullivant2023algebraic}.

\begin{definition}
A \emph{simplicial complex} on a finite ground set $[m]$ is a collection
$\Gamma \subseteq 2^{[m]}$ such that for all $F \in \Gamma$, every subset
$F' \subseteq F$ also belongs to $\Gamma$.
A \emph{face} of $\Gamma$ is any element of $\Gamma$.
A \emph{facet} of $\Gamma$ is a maximal face of $\Gamma$ with respect to inclusion.
A simplicial complex $\Gamma$ is \emph{reducible} with decomposition
$(\Gamma_1, B, \Gamma_2)$ if $\Gamma_1 \cup \Gamma_2 = \Gamma$,
$\Gamma_1 \cap \Gamma_2 = 2^B$, and $\Gamma_1,\Gamma_2 \neq 2^B$.
Finally, $\Gamma$ is \emph{decomposable} if it is reducible and both components
$\Gamma_1$ and $\Gamma_2$ are either decomposable or simplices (of the form
$2^F$ for some $F$).
\end{definition}

\begin{figure}[ht]
\centering

\begin{minipage}{0.45\textwidth}
\centering

\vspace{0.3cm}

\begin{tikzpicture}[scale=1]
\fill (0,0) circle (2pt) node[below] {$1$};
\fill (2,0) circle (2pt) node[below] {$2$};
\fill (1,1.5) circle (2pt) node[above] {$3$};
\fill (3,1.5) circle (2pt) node[above] {$4$};

\draw (0,0)--(2,0)--(1,1.5)--cycle;
\draw (2,0)--(3,1.5)--(1,1.5)--cycle;
\end{tikzpicture}

\vspace{0.2cm}

Facets: $\{1,2,3\}$ and $\{2,3,4\}$
\end{minipage}
\hfill
\begin{minipage}{0.45\textwidth}
\centering

\vspace{0.3cm}

\begin{tikzpicture}[scale=1]
\fill (0,0) circle (2pt) node[below left] {$1$};
\fill (2,0) circle (2pt) node[below right] {$2$};
\fill (2,2) circle (2pt) node[above right] {$3$};
\fill (0,2) circle (2pt) node[above left] {$4$};

\draw (0,0)--(2,0)--(2,2)--(0,2)--cycle;
\end{tikzpicture}

\vspace{0.2cm}

Facets: $\{1,2\}, \{2,3\}, \{3,4\}, \{1,4\}$
\end{minipage}

\caption{Examples of decomposable (left) and nondecomposable (right) simplicial complexes.}
\label{fig:decomposable_complexes}
\end{figure}

Examples of decomposable  and nondecomposable  simplicial complexes are depicted in Figure~\ref{fig:decomposable_complexes}. We briefly recall the notion of hierarchical models.
\begin{definition} \label{def:hierarchical_model}
Let $\textbf{X} = (X_1, \ldots, X_m)$ be an $m$-dimensional discrete random vector, where each
variable $X_i$ has finite state space $[r_i]$.
Let $\textstyle{R=\prod_{i=1}^m [r_i]}$ denote the joint state space, and fix a simplicial
complex $\Gamma$ on the ground set $[m]$.
The \emph{hierarchical model} associated with $\Gamma$ consists of all probability
distributions $(p_i)_{i\in R}$ on $R$ that admit a parametrization of the form
\begin{equation}
\label{eq:hierarchical model}
p_i
=
\frac{1}{Z(\theta)}
\prod_{F \text{ facet of } \Gamma}
\theta^{(F)}_{i_F},
\qquad i\in R,
\end{equation}
where $\theta^{(F)}_{i_F}$ are nonnegative parameters associated with the facet $F$,
$i_F$ denotes the restriction of $i$ to the coordinates indexed by $F$,
and $Z(\theta)$ is a normalizing constant ensuring that $\textstyle{\sum_{i\in R} p_i = 1}$.
The hierarchical model is said to be \emph{decomposable} if the simplicial complex
$\Gamma$ is decomposable.
\end{definition}

The following lemma shows that the general nonhomogeneous higher-order multistate
Markov model fits naturally into this framework. Recall that a path in the model is represented by a sequence $(i_1,\ldots,i_n)$ of states of length $n$.

\begin{lemma}
\label{lem:markov_is_hierarchical}
Let $\phi$ be the parametrization map of a nonhomogeneous $k$-th order multistate Markov model. The set $\phi \left(\textstyle{\mathbb{R}^{S^k} \times \prod_{\ell=k+1}^n \mathbb{R}^{S^{k+1}}}
\right) \cap \Delta_{|S|^n-1}$ is a decomposable
hierarchical model as defined in Definition~\ref{def:hierarchical_model}.
\end{lemma}
\begin{proof}
Consider the parametrization map $\textstyle{p_{i_1 \ldots i_n}= \pi_{i_1 \ldots i_k}\prod_{j=k+1}^n a^{(\ell)}_{i_{\ell-k}\ldots i_{\ell}}}$ for a nonhomogeneous $k$-th order Markov multistate model on the state space $S$.
Here $(i_1,\ldots,i_n)$ denotes an arbitrary path of length $n$ in the model and the argument below does not depend on the particular choice of path.

Define a simplicial complex $\Gamma$ on $[n]$ whose facets are
the $(k+1)$-subsets
$\{1,\ldots,k+1\}$, $\{2,\ldots,k+2\},\ldots$, $\{n-k,\ldots,n\}$.
This simplicial complex is decomposable since it admits the recursive decomposition obtained by
splitting off the last facet $\{n-k,\ldots,n\}$ along the separator $\{n-k,\ldots,n-1\}$,
and iterating.

We relabel the parameters in~\eqref{eq:hierarchical model} by setting $\theta^{(\{1,\ldots,k+1 \})}_{i_{\ell-k}\ldots i_{\ell}} := \pi_{i_1 \ldots i_k} a^{(k+1)}_{i_{1}\ldots i_{k+1}} $ and
$\theta^{(\{\ell-k,\ldots,\ell\})}_{i_{\ell-k}\ldots i_{\ell}} := a^{(\ell)}_{i_{\ell-k}\ldots i_{\ell}} $ for
$\ell=k+2,\ldots,n$.
With these identifications, the product over facets in~\eqref{eq:hierarchical model}
equals $\textstyle{\prod_{\ell=k+1}^n \theta^{(\{\ell-k,\ldots,\ell\})}_{i_{\ell-k}\ldots i_{\ell}}}= \pi_{i_1 \ldots i_k} \textstyle{\prod_{\ell=k+1}^n a^{(\ell)}_{i_{\ell-k}\ldots i_{\ell}}} = p_{i_1\ldots i_n} $ (up to the normalizing constant). 
Hence the set $\phi \left(\textstyle{\mathbb{R}^{S^k} \times \prod_{\ell=k+1}^n \mathbb{R}^{S^{k+1}}}
\right) \cap \Delta_{|S|^n-1}$ is a hierarchical model for $\Gamma$, as desired. 
\end{proof}

Lemma~\ref{lem:markov_is_hierarchical} does not generalize to homogeneous models because the parameters corresponding to different simplices in the simplicial complex are identified, so a homogeneous $k$-th order multistate model is not a decomposable hierarchical model in the sense of Definition~\ref{def:hierarchical_model}. Instead, it is a subset of the decomposable hierarchical model from Lemma~\ref{lem:markov_is_hierarchical}.

From an algebraic perspective, forbidding transitions corresponds to setting certain joint path
probabilities identically equal to zero.
The following proposition makes this precise.
\begin{proposition}
\label{prop:restricted_model_slice}
Let $F$ denote the set of paths
$(i_1, \ldots, i_n)$  involving at least one forbidden transition. 
Then $\mathcal{V}_F$ can be identified with the coordinate slice
\[
\mathcal{V}_F
=
\mathcal{V}
\cap
\{\,p \mid p_{i_1 \ldots i_n}=0
\text{ for all }  (i_1, \ldots, i_n)\in F  \,\}.
\]
\end{proposition}
\begin{proof}
For the inclusion ``$\subseteq$'', let $p\in\mathcal{V}_F$.
By construction, $p$ arises from the parametrization of $\mathcal{V}$ with all
forbidden transition parameters set to zero.
Hence $p$ lies in $\mathcal{V}$ and satisfies $p_{i_1 \ldots i_n}=0$ for all
$(i_1, \ldots, i_n)\in F$, proving the left-hand inclusion.

For the inclusion ``$\supseteq$'', let $p$ lie in the right-hand side.
Then $p$ admits a parametrization in $\mathcal{V}$, possibly with nonzero values
assigned to parameters corresponding to forbidden transitions.
Since $p_{i_1 \ldots i_n}=0$ for all
 $(i_1, \ldots, i_n)\in F$, these
parameters may be set to zero without changing $p$.
The resulting parameter choice lies in the parameter space of $\mathcal{V}_F$,
so $p\in\mathcal{V}_F$.
\end{proof}

\subsection{Implicit description} \label{sec:implicit_description_nonhomogeneous_subsection}

We now describe the vanishing ideal of the general nonhomogeneous $k$-th order multistate Markov model when the statistical constraints on the parameters are disregarded.
Since this model is hierarchical as shown in Section~\ref{sec:hierarchical_models}, its vanishing ideal is generated by binomials corresponding to
conditional independence statements between appropriate blocks of variables.

\begin{proposition}
\label{prop:generating set general}
The vanishing ideal $\ker(\phi^*)$ of the general $k$-th order multistate Markov model is generated by polynomials of the form
\begin{align*}
&p_{i_1\ldots i_r,\, j_{r+1}\ldots j_{r+k},\, s_{r+k+1}\ldots s_n}\,
 p_{i'_1\ldots i'_r,\, j_{r+1}\ldots j_{r+k},\, s'_{r+k+1}\ldots s'_n} \\
&\qquad -
p_{i_1\ldots i_r,\, j_{r+1}\ldots j_{r+k},\, s'_{r+k+1}\ldots s'_n}\,
 p_{i'_1\ldots i'_r,\, j_{r+1}\ldots j_{r+k},\, s_{r+k+1}\ldots s_n},
\end{align*}
or, equivalently, in block notation,
\[
p_{IJS}p_{I'JS'} - p_{IJS'}p_{I'JS},
\]
where $|I|=|I'|=r$, $|J|=k$, and $|S|=|S'|=n-k-r$. Here
\[
I=(i_1,\ldots,i_r), \
I'=(i'_1,\ldots,i'_r), \
J=(j_{r+1},\ldots,j_{r+k}), \
S=(s_{r+k+1},\ldots,s_n), \
S'=(s'_{r+k+1},\ldots,s'_n),
\]
where $I$, $I'$, $J$, $S$, and $S'$ correspond to consecutive blocks of positions in a path of length $n$.
These binomials correspond to exchanging the values of the blocks $I$ and $I'$ while keeping the
conditioning block $J$ fixed.
\end{proposition}

\begin{example}
To illustrate the notation in Proposition~\ref{prop:generating set general}, consider the nonhomogeneous illness--death model from Example~\ref{example:illness_death_model} with $k=1$ and paths of length $n=4$. For this model, the generators of the vanishing ideal can be written in the form
\[
p_{IJS}p_{I'JS'}-p_{IJS'}p_{I'JS},
\]
where $|I|=|I'|=r$ for $r=1$ or $r=2$, and $|J|=1$ and $|S|=|S'|=4-r-1$. Table~\ref{tab:illness-death-generators} shows how these generators correspond to the notation of Proposition~\ref{prop:generating set general}.
\begin{table}[h]
\centering
\caption{Generators of the nonhomogeneous ideal in Example~\ref{example:illness_death_model} in the notation of Proposition~\ref{prop:generating set general}. The generators of the vanishing ideal can be written in the form
$p_{IJS}p_{I'JS'}-p_{IJS'}p_{I'JS},$
where $|I|=|I'|=r$ for $r=1$ or $r=2$, and $|J|=1$ and $|S|=|S'|=4-r-1$. }
\label{tab:illness-death-generators}
\[
\begin{array}{c|c|c|c|c|c}
\text{Binomial} & I & I' & J & S & S' \\
\hline
p_{1112}p_{0122}-p_{1122}p_{0112}
& (1) & (0) & (1) & (1,2) & (2,2) \\[1mm]
p_{1111}p_{0122}-p_{1122}p_{0111}
& (1) & (0) & (1) & (1,1) & (2,2) \\[1mm]
p_{1111}p_{0112}-p_{1112}p_{0111}
& (1) & (0) & (1) & (1,1) & (1,2)  \\[1mm]
p_{1111}p_{0012}-p_{1112}p_{0011}
& (1,1) & (0,0) & (1) & (1) & (2) \\[1mm]
p_{0111}p_{0012}-p_{0112}p_{0011}
& (0,1) & (0,0) & (1) & (1) & (2) 
\end{array}
\]
\end{table} 
\end{example}

The result is a special case of a general result about vanishing ideals for decomposable hierarchical models~\cite[Theorem 4.5]{dobra2003markov}.
\begin{remark}\label{rem:nonhom-subideal}
We give a statistical interpretation of why polynomials in Proposition~\ref{prop:generating set general} vanish on the model, by rewriting the polynomial relations as an equality of ratios and verifying this equality using Bayes' formula.  When setting the polynomial in the statement of Proposition~\ref{prop:generating set general} to zero, moving the second term to the right-hand side and then dividing both sides by the second factors (assuming that they are nonzero), we get
\begin{align*}
\frac{p_{i_1\ldots i_r,\, j_{r+1}\ldots j_{r+k},\, s_{r+k+1}\ldots s_n}}{p_{i_1\ldots i_r,\, j_{r+1}\ldots j_{r+k},\, s'_{r+k+1}\ldots s'_n}} = 
\frac{p_{i'_1\ldots i'_r,\, j_{r+1}\ldots j_{r+k},\, s_{r+k+1}\ldots s_n}}{p_{i'_1\ldots i'_r,\, j_{r+1}\ldots j_{r+k},\, s'_{r+k+1}\ldots s'_n}}.
\end{align*}
In terms of path probabilities, the equality of ratios reads:
\begin{align*}
&\qquad\frac{
P\!\left(
X_1=i_1,\ldots,X_r=i_r,\,
X_{r+1}=j_{r+1},\ldots,X_{r+k}=j_{r+k},\,
X_{r+k+1}=s_{r+k+1},\ldots,X_n=s_n
\right)}
{
P\!\left(
X_1=i_1,\ldots,X_r=i_r,\,
X_{r+1}=j_{r+1},\ldots,X_{r+k}=j_{r+k},\,
X_{r+k+1}=s'_{r+k+1},\ldots,X_n=s'_n
\right)}
\\
&\qquad=\frac{
P\!\left(
X_1=i'_1,\ldots,X_r=i'_r,\,
X_{r+1}=j_{r+1},\ldots,X_{r+k}=j_{r+k},\,
X_{r+k+1}=s_{r+k+1},\ldots,X_n=s_n
\right)}
{
P\!\left(
X_1=i'_1,\ldots,X_r=i'_r,\,
X_{r+1}=j_{r+1},\ldots,X_{r+k}=j_{r+k},\,
X_{r+k+1}=s'_{r+k+1},\ldots,X_n=s'_n
\right)}.
\end{align*}

By Bayes' formula, the first ratio in the above equality simplifies as follows. 
\begin{align*} &
\frac{
P\!\left(
X_1=i_1,\ldots,X_r=i_r,\,
X_{r+1}=j_{r+1},\ldots,X_{r+k}=j_{r+k},\,
X_{r+k+1}=s_{r+k+1},\ldots,X_n=s_n
\right)}
{
P\!\left(
X_1=i_1,\ldots,X_r=i_r,\,
X_{r+1}=j_{r+1},\ldots,X_{r+k}=j_{r+k},\,
X_{r+k+1}=s'_{r+k+1},\ldots,X_n=s'_n
\right)} \\ 
& \qquad = \frac{
P\!\left(
X_{r+k+1}=s_{r+k+1},\ldots,X_n=s_n
\ \middle|\  X_1 = i_1, \ldots X_r = i_r,
X_{r+1}=j_{r+1},\ldots,X_{r+k}=j_{r+k}
\right)}
{
P\!\left(
X_{r+k+1}=s'_{r+k+1},\ldots,X_n=s'_n
\ \middle|\  X_1 = i_1, \ldots, X_r = i_r,
X_{r+1}=j_{r+1},\ldots,X_{r+k}=j_{r+k}
\right)} & \\ 
 & \qquad=
\frac{
P\!\left(
X_{r+k+1}=s_{r+k+1},\ldots,X_n=s_n
\ \middle|\ 
X_{r+1}=j_{r+1},\ldots,X_{r+k}=j_{r+k}
\right)}
{
P\!\left(
X_{r+k+1}=s'_{r+k+1},\ldots,X_n=s'_n
\ \middle|\ 
X_{r+1}=j_{r+1},\ldots,X_{r+k}=j_{r+k}
\right)}, \; \text{(Markov property).} &
\end{align*}
Similarly, applying Bayes' formula for the second ratio, we get the same expression, which shows that the two ratios are equal. This confirms that the generators of Proposition~\ref{prop:generating set general} vanish on the model, giving a direct probabilistic explanation; however, since it does not show that they generate the full vanishing ideal, this is not an alternative proof of Proposition~\ref{prop:generating set general}. 
\end{remark}

The description in Proposition~\ref{prop:generating set general} easily extends to models in which some transitions are forbidden.

\begin{corollary}
\label{lem:generators_restricted}
With the notation of Proposition~\ref{prop:restricted_model_slice}, the vanishing
ideal of $\mathcal{V}_F$ is obtained from that of $\mathcal{V}$ by adding the polynomials
$p_{i_1 \ldots i_n}$  for all $(i_1,\ldots,i_n) \in F$. 
\end{corollary}

\if 0
\begin{lemma}
\label{lem:generators_restricted}
Let $\mathcal{M}$ be the general $k$-th order multistate Markov model, and let
$\mathcal{M}_F$ be a model obtained by forbidding a specified collection of transitions.
Let $P_F$ denote the set of joint probability variables
$p_{i_1 \ldots i_n}$ corresponding to forbidden transitions.
Then the vanishing ideal of $\mathcal{M}_F$ is obtained from that of $\mathcal{M}$ by
adding the linear equations $p_{i_1 \ldots i_n}=0$ for all
$p_{i_1 \ldots i_n}\inP_F$ and eliminating the variables in $P_F$.
\end{lemma}
\fi

\begin{proof}
As discussed in the previous subsection, $\mathcal{M}_F$ is a slice of the general model $\mathcal{M}$. 
The vanishing ideal of the intersection of two algebraic sets is the sum of their vanishing ideals. 
In the present setting, this implies that the generators of the restricted model are obtained by 
adding the polynomials   $p_{i_1 \ldots i_n}$  for all $(i_1,\ldots,i_n) \in F$ to the generating set of the general model. 
\end{proof}

In particular, Corollary~\ref{lem:generators_restricted} implies that no new higher-degree generators arise from forbidden transitions beyond those
already present in the unrestricted model.

\begin{remark}
Note that Proposition~\ref{prop:restricted_model_slice} and  Corollary~\ref{lem:generators_restricted}
apply equally to homogeneous multistate Markov models, since their proofs
do not use the nonhomogeneous assumption. 
\end{remark}

\subsection{Inside probability simplex} \label{sec:summing_to_one_condition}

In this section we show that $\mathcal{M} =\mathcal{V} \cap \Delta_{|S|^n-1}$ in the nonhomogeneous setting. Here $\Theta$ is a set of parameters satisfying the constraints \eqref{constraints:korder}.

\begin{proposition}\label{prop:equal_ideal_with_constraints}
We have $\phi(\Theta) = \phi \left(\mathbb{R}^{S^k} \times \prod_{\ell=k+1}^n \mathbb{R}^{S^{k+1}}\right)  \cap \Delta_{|S|^n-1}$.
\end{proposition}

\begin{proof}
Firstly, 
\[
\phi(\Theta) \subseteq \phi \left(\mathbb{R}^{S^k} \times \prod_{\ell=k+1}^n \mathbb{R}^{S^{k+1}}\right)  \cap \Delta_{|S|^n-1},
\]
since $\Theta\subseteq \mathbb{R}^{S^k} \times \prod_{\ell=k+1}^n \mathbb{R}^{S^{k+1}}$ and the statistical constraints on the parameters guarantee that the image is contained in $\Delta_{|S|^n-1}$. 

To prove the reverse inclusion, let $p=(p_{i_1\ldots i_n})\in \phi (\mathbb{R}^{S^k} \times \prod_{\ell=k+1}^n \mathbb{R}^{S^{k+1}})  \cap \Delta_{|S|^n-1} $. More precisely, 
\[
p_{i_1 \ldots i_n} \geq 0, \,  \sum_{i_1,\ldots,i_n}p_{i_1 \ldots i_n} = 1, \text{ and } p_{i_1\ldots i_n}=\pi_{i_1\ldots i_k}\prod_{\ell=k+1}^n a^{(\ell)}_{i_{\ell-k}\cdots i_{\ell-1} i_{\ell}},
\]
where all the parameters are real numbers, but $(\pi_{i_1\ldots i_k},a^{(\ell)}_{i_{\ell-k}\ldots i_\ell})$ does not necessarily lie in $\Theta$.

We now construct new parameters
$(\pi'_{i_1\ldots i_k},a'^{(\ell)}_{i_{\ell-k}\ldots i_\ell})\in \Theta$, that are mapped to $(p_{i_1\ldots i_n})$ with $\phi$.

First, let
$b^{(k+1)}_{i_1\ldots i_{k+1}} = \left| \pi_{i_1\ldots i_k}a^{(k+1)}_{i_1\ldots i_{k+1}} \right|$, and $b^{(\ell)}_{i_{\ell-k}\ldots i_{\ell}}=\left|a^{(\ell)}_{i_{\ell-k}\ldots i_{\ell}} \right|$ for $\ell\geq k+2$. 
Then, \[
p_{i_1\ldots i_n}=\prod_{\ell=k+1}^n b^{(\ell)}_{i_{\ell-k}\cdots i_{\ell-1} i_{\ell}}.
\]

To simplify the notation in what follows, for a fixed path $(i_1, \ldots, i_n)$ and $\ell\geq 1$ denote
$$S_{i_1\ldots i_\ell} = \{i_1\}\times \{i_2\}\times \ldots \times \{i_\ell\} \times S^{n-\ell},$$
that is, the set of all paths whose first $\ell$ states are fixed.

Since the transition probabilities are allowed to depend on the time step $\ell$ in the nonhomogeneous setting, we may reconstruct the conditional
probabilities independently for each time level. 
Define the parameters as
\[
a'^{(\ell)}_{i_{\ell-k}\ldots i_\ell} = \frac{\sum_{(j_1, \ldots, j_n)\in S_{i_1\ldots i_{\ell}}} 
\prod_{t = k+1}^n b^{(t)}_{j_{t-k}\ldots  j_t}}{\sum_{(j_1, \ldots, j_n)\in S_{i_1\ldots i_{\ell-1}}} 
\prod_{t = k+1}^n b^{(t)}_{j_{t-k}\ldots  j_t}}, \quad 
\pi'_{i_1\ldots i_k} = \frac{\sum_{(j_1, \ldots, j_n)\in S_{i_1\ldots i_k}}\prod_{t=k+1}^n b^{(t)}_{j_{t-k}\ldots j_t}}{\sum_{(j_1,\ldots, j_n) \in S^n} \prod_{t=k+1}^n b^{(t)}_{j_{t-k}\ldots j_t}}.
\]
Since $b^{(\ell)}_{i_{\ell-k}\cdots i_{\ell}}$ were defined as absolute values, all parameters above are nonnegative. Furthermore,
\[
\sum_{i_\ell \in S}
a'^{(\ell)}_{i_{\ell-k}\cdots i_\ell}= \sum_{i_\ell \in S} \frac{\sum_{(j_1, \ldots, j_n)\in S_{i_1\ldots i_{\ell}}} 
\prod_{t = k+1}^n b^{(t)}_{j_{t-k}\ldots  j_t}}{\sum_{(j_1, \ldots, j_n)\in S_{i_1\ldots i_{\ell-1}}} 
\prod_{t = k+1}^n b^{(t)}_{j_{t-k}\ldots  j_t}} =  \frac{\sum_{(j_1, \ldots, j_n)\in S_{i_1\ldots i_{\ell-1}}} 
\prod_{t = k+1}^n b^{(t)}_{j_{t-k}\ldots  j_t}}{\sum_{(j_1, \ldots, j_n)\in S_{i_1\ldots i_{\ell-1}}} 
\prod_{t = k+1}^n b^{(t)}_{j_{t-k}\ldots  j_t}} = 1,
\]
$\forall \ell \in \{k+1,\ldots,n\}, \, \forall i_{\ell-k},\ldots,i_{\ell-1} \in S$,
and
\[
\sum_{(i_1,\ldots,i_k) \in S^k}\pi'_{i_1\cdots i_k} =\sum_{(i_1,\ldots,i_k) \in S^k} \frac{\sum_{(j_1, \ldots, j_n)\in S_{i_1\ldots i_k}}\prod_{t=k+1}^n b^{(t)}_{j_{t-k}\ldots j_t}}{\sum_{(j_1,\ldots, j_n) \in S^n} \prod_{t=k+1}^n b^{(t)}_{j_{t-k}\ldots j_t}}= \frac{\sum_{(j_1,\ldots, j_n) \in S^n} \prod_{t=k+1}^n b^{(t)}_{j_{t-k}\ldots j_t}}{\sum_{(j_1,\ldots, j_n) \in S^n} \prod_{t=k+1}^n b^{(t)}_{j_{t-k}\ldots j_t}}=1.
\]
Hence $(\pi',a')\in\Theta$.

To show that the new parameters give the required parametrization, we will first note some of their properties. The denominator of $\pi'_{i_1\ldots i_k}$ is 1, as it is the sum of all coordinates $(p_{i_1\ldots i_n})$. For any path $(i_1, \ldots, i_n)$, the numerator of $\pi'_{i_1\ldots i_k}$ coincides with the denominator of $a'^{(k+1)}_{i_1\ldots i_{k+1}}$.
Moreover, the numerator of $a'^{(\ell)}_{i_{\ell-k}\cdots i_{\ell}}$ coincides with the denominator of the next step transition probability $a'^{(\ell+1)}_{i_{\ell-k+1}\cdots i_{\ell+1}}$. Thus, the product
$\textstyle{\pi'_{i_1\ldots i_k} \prod_{\ell=k+1}^n a'^{(\ell)}_{i_{\ell-k}\cdots i_{\ell-1} i_{\ell}}}$ is telescoping: the denominator of each factor coincides with the numerator of the preceding factor, and the only term that is left is the numerator of the last transition probability. This allows us to show the required factorization:
\[\begin{split} \pi'_{i_1\ldots i_k} \prod_{\ell=k+1}^n a'^{(\ell)}_{i_{\ell-k}\cdots i_{\ell-1} i_{\ell}}&=
\frac{\sum_{(j_1, \ldots, j_n)\in S_{i_1\ldots i_k}}\prod_{t=k+1}^n b^{(t)}_{j_{t-k}\ldots j_t}}{\sum_{(j_1,\ldots, j_n) \in S^n} \prod_{t=k+1}^n b^{(t)}_{j_{t-k}\ldots j_t}}\prod_{\ell=k+1}^n \frac{\sum_{(j_1, \ldots, j_n)\in S_{i_1\ldots i_{\ell}}} 
\prod_{t = k+1}^n b^{(t)}_{j_{t-k}\ldots  j_t}}{\sum_{(j_1, \ldots, j_n)\in S_{i_1\ldots i_{\ell-1}}} 
\prod_{t = k+1}^n b^{(t)}_{j_{t-k}\ldots  j_t}} \\
&=\underbrace{\frac{\sum_{(j_1, \ldots, j_n)\in S_{i_1\ldots i_k}}\prod_{t=k+1}^n b^{(t)}_{j_{t-k}\ldots j_t}}{1}}_{\pi'_{i_1\ldots i_k}}
\underbrace{\frac{\sum_{(j_1, \ldots, j_n)\in S_{i_1\ldots i_{k+1}}} 
\prod_{t = k+1}^n b^{(t)}_{j_{t-k}\ldots  j_t}}{\sum_{(j_1, \ldots, j_n)\in S_{i_1\ldots i_{k}}}\prod_{t = k+1}^n b^{(t)}_{j_{t-k}\ldots  j_t}}}_{a'^{(k+1)}_{i_{1}\cdots i_{k+1}}}\\
&\ldots \underbrace{\frac{\sum_{(j_1,\ldots,j_n)\in S_{i_1\ldots i_{n-1}}} 
\prod_{t = k+1}^n b^{(t)}_{j_{t-k}\ldots  j_t}}{\sum_{(j_1, \ldots, j_n)\in S_{i_1\ldots i_{n-2}}} 
\prod_{t = k+1}^n b^{(t)}_{j_{t-k}\ldots  j_t}}}_{a'^{(n-1)}_{i_{n-k-1}\cdots i_{n-1}}}  
\underbrace{\frac{\prod_{t = k+1}^n b^{(t)}_{j_{t-k}\ldots  j_t}}{\sum_{(j_1,\ldots,j_n)\in S_{i_1\ldots i_{n-1}}} 
\prod_{t = k-1}^n b^{(t)}_{j_{t-k}\ldots  j_t}}}_{a'^{(n)}_{i_{n-k}\cdots i_{n}}}  \\
&=\prod_{t=k+1}^n b^{(t)}_{i_{t-k}\cdots i_{t-1} i_{t}}=p_{i_1\ldots i_n} .\qedhere
\end{split}
\]

\end{proof}

The proof does not extend to the homogeneous case because the definitions of $a'^{(\ell)}_{i_{\ell-k}\ldots i_\ell}$ and $\pi'_{i_1\ldots i_k}$ do not generalize.

\begin{remark} \label{rmk:nh_msm_as_graphical_models}
The $k$-th order nonhomogeneous Markov chain can also be seen as a discrete directed graphical model on $n$ vertices with an edge $i \rightarrow \ell$ for every $\ell \in \{k+1,\ldots,n\}$ and $i \in \{\ell-k,\ldots,\ell-1\}$. Proposition~\ref{prop:equal_ideal_with_constraints} follows also from~\cite[Theorem 3.27]{lauritzen1996graphical}.
\end{remark}

The statement of Proposition~\ref{prop:equal_ideal_with_constraints} extends verbatim to models with
forbidden transitions by intersecting both sides with the corresponding coordinate hyperplanes. On the other hand, although we have equality $I(\mathcal{M})=\textstyle{\sqrt{I(\mathcal{V}) + \left\langle 1-\sum_{i_1,\ldots, i_n}p_{i_1\ldots i_n} \right\rangle}}$ over the complex numbers~$\mathbb{C}$, we do not know whether this equality also holds
over the real numbers~$\mathbb{R}$.

\section{Homogeneous multistate Markov  models}
\label{sec:hom}

In the nonhomogeneous setting, multistate Markov models when the statistical constraints on the parameters are disregarded admit a clean algebraic description:
they are decomposable hierarchical models, and their vanishing ideals are generated by explicit
families of binomials arising from local Markov moves
(Section~\ref{sec:implicitization_nonhomogeneous}).
Imposing time homogeneity fundamentally changes this picture.
Although homogeneous multistate Markov models are still defined by monomial parametrizations,
the identification of transition parameters across time induces additional polynomial relations
among joint path probabilities.
As a consequence, homogeneous models generally fall outside the class of decomposable hierarchical
models, and their vanishing ideals are strictly larger than those of their nonhomogeneous counterparts.

This difference already appears in simple examples.
For instance, in the homogeneous illness-death model
(see column Prob (homogeneous), Table~\ref{tab:obsevations4}),
the elimination ideal contains binomial relations that are not implied
by the Markov property alone.
For example, while the binomial 
\[
p_{0012}p_{0111}-p_{0011}p_{0112}
\]
appears in the nonhomogeneous case, the additional
relations
\[
p_{0011}p_{0012}-p_{0001}p_{0112}
\qquad\text{and}\qquad
p_{0011}^2-p_{0001}p_{0111}
\]
reflect the imposed equality of transition probabilities across time points.

From an algebraic~perspective, these additional relations arise from identifying transition
parameters associated with different time indices.
Several families of binomial relations for homogeneous models can be derived directly from the
parametrization by exploiting shift invariance of the transition probabilities.
In particular, odds-ratio–type relations analogous to those in the nonhomogeneous case
vanish on all homogeneous Markov chains.
However, explicit computations show that these relations do not, in general, generate the full
vanishing ideal $\ker(\phi^*)$.

\begin{remark}
The homogeneous model is obtained from the corresponding nonhomogeneous model by imposing that transition probabilities do not depend on the time index. More precisely to pass from the nonhomogeneous model to the homogeneous model, in the parametrization of Section~\ref{sec:algebraic_description}, the parameters
$a^{(\ell)}_{i_{\ell-k}\cdots i_{\ell-1}i_\ell}$
are all replaced by
$a_{i_{\ell-k}\cdots i_{\ell-1}i_\ell}$ for all admissible time indices~$\ell$.
Consequently, every polynomial relation among the joint path probabilities that holds for the nonhomogeneous model also holds for the homogeneous model. In particular, the binomial generators described in Proposition~\ref{prop:generating set general} vanish on every homogeneous multistate Markov model.
Time-homogeneity imposes additional algebraic constraints. However, as illustrated by Example~\ref{example:illness_death_model}, these relations do not, in general, generate the entire vanishing ideal. Thus the homogeneous vanishing ideal should be viewed as an enlargement of the nonhomogeneous vanishing ideal obtained by imposing time homogeneity and all algebraic consequences of this identification. Proposition~\ref{prop:homogeneous binomials} provides a general family of binomial relations arising from these parameter identifications. 
\end{remark}

As a first step toward understanding the algebraic structure of homogeneous models,
we identify a general family of binomial relations that necessarily belong to the vanishing
ideal.
The following proposition provides a homogeneous analogue of
Proposition~\ref{prop:generating set general}.
  
\begin{proposition}
\label{prop:homogeneous binomials}
Let $\mathcal M$ be the $k$-th order homogeneous multistate Markov chain.
Fix two states $\psi,\eta$ and two consecutive blocks of positions in a path of length $k$
\[
\Gamma=(\gamma_1,\ldots,\gamma_k), \qquad
\Delta=(\delta_1,\ldots,\delta_k).
\]
Then the vanishing ideal $I(\mathcal M)$ contains all binomials of the form
\[
p_{I,\,\Gamma,\,\psi,\,\Delta,\,J}\,
p_{I',\,\Gamma,\,\eta,\,\Delta,\,J'}
-
p_{I,\,\Gamma,\,\eta,\,\Delta,\,J}\,
p_{I',\,\Gamma,\,\psi,\,\Delta,\,J'},
\]
where $I,I',J,J'$ are index blocks such that the concatenations
\[
I\,\Gamma\,\psi\,\Delta\,J,\quad
I'\,\Gamma\,\psi\,\Delta\,J',\quad
I\,\Gamma\,\eta\,\Delta\,J,\quad
\text{and}\quad
I'\,\Gamma\,\eta\,\Delta\,J'
\]
define paths of length $n$.
If one of the blocks $I$ or $I'$ is empty (respectively, $J$ or $J'$ is empty),
then the corresponding block $\Gamma$ (respectively, $\Delta$) 
may have length strictly smaller than $k$.
\end{proposition}

\begin{proof}
Assume $I,J,I',J'$ are nonempty (the empty-block case is identical). Fix states
$\psi,\eta$ and fixed states $\gamma_1,\dots,\gamma_k$ and $\delta_1,\dots,\delta_k$.
We first treat the interior case, i.e.\ when the symbol $\psi$ (or $\eta$) occurs at a position
$r$ with $k+1\le r,\ell\le n-k$, so that the
corresponding joint-probability entries are
\[
p_{i_1,\dots,i_{r-k-1},\gamma_1,\dots,\gamma_k,\psi,\delta_1,\dots,\delta_k,i_{r+k+1},\dots,i_n}
\quad\text{and}\quad
p_{j_1,\dots,j_{\ell-k-1},\gamma_1,\dots,\gamma_k,\psi,\delta_1,\dots,\delta_k,j_{\ell+k+1},\dots,j_n},
\]
and likewise with $\psi$ replaced by $\eta$.

\medskip

Since we are in the \emph{homogeneous} $k$-th order Markov model, the joint probability of a path
factors as
\[
p_{x_1,\dots,x_n}
=\pi_{x_1,\dots,x_k}\,\prod_{t=k+1}^{n} a_{x_{t-k}\cdots x_{t-1}x_t},
\]
where the transition parameters $a_{x_{t-k}\cdots x_{t-1}x_t}$ do \emph{not} depend on $t$.

Consider the ratio of the two path probabilities that differ only by $\psi$ versus $\eta$ at the
position $r$. By cancellation of all common factors in the above product, we get
\begin{align*}
\frac{
p_{i_1,\dots,i_{r-k-1},\gamma_1,\dots,\gamma_k,\psi,\delta_1,\dots,\delta_k,i_{r+k+1},\dots,i_n}
}{
p_{i_1,\dots,i_{r-k-1},\gamma_1,\dots,\gamma_k,\eta,\delta_1,\dots,\delta_k,i_{r+k+1},\dots,i_n}
}
&=
\frac{
a_{\gamma_1\cdots\gamma_k\psi}\,
a_{\gamma_2\cdots\gamma_k\psi\delta_1}\cdots
a_{\psi\delta_1\cdots\delta_k}
}{
a_{\gamma_1\cdots\gamma_k\eta}\,
a_{\gamma_2\cdots\gamma_k\eta\delta_1}\cdots
a_{\eta\delta_1\cdots\delta_k}
}.
\end{align*}
Exactly the same cancellation argument for the ratio at the position $\ell$ yields
\begin{align*}
\frac{
p_{j_1,\dots,j_{\ell-k-1},\gamma_1,\dots,\gamma_k,\psi,\delta_1,\dots,\delta_k,j_{\ell+k+1},\dots,j_n}
}{
p_{j_1,\dots,j_{\ell-k-1},\gamma_1,\dots,\gamma_k,\eta,\delta_1,\dots,\delta_k,j_{\ell+k+1},\dots,j_n}
}
&=
\frac{
a_{\gamma_1\cdots\gamma_k\psi}\,
a_{\gamma_2\cdots\gamma_k\psi\delta_1}\cdots
a_{\psi\delta_1\cdots\delta_k}
}{
a_{\gamma_1\cdots\gamma_k\eta}\,
a_{\gamma_2\cdots\gamma_k\eta\delta_1}\cdots
a_{\eta\delta_1\cdots\delta_k}
}.
\end{align*}
Hence, the two ratios are equal.
Equivalently, by Bayes' formula, the equality holds
if and only if
\begin{align*}
P\!\left(
X_r=\psi \mid 
X_1=i_1,\ldots,X_{r-k-1}=i_{r-k-1},\,
X_{r-k}=\gamma_1,\ldots,X_{r-1}=\gamma_k, \right. \\
\left.
X_{r+1}=\delta_1,\ldots,X_{r+k}=\delta_k,\,
X_{r+k+1}=i_{r+k+1},\ldots,X_n=i_n
\right)
\\
=
P\!\left(
X_\ell=\psi \mid 
X_1=j_1,\ldots,X_{\ell-k-1}=j_{\ell-k-1},\,
X_{\ell-k}=\gamma_1,\ldots,X_{\ell-1}=\gamma_k, \right. \\
\left.
X_{\ell+1}=\delta_1,\ldots,X_{\ell+k}=\delta_k,\,
X_{\ell+k+1}=j_{\ell+k+1},\ldots,X_n=j_n
\right)
\\
=
P\!\left(
X_r=\psi \mid
X_{r-k}=\gamma_1,\ldots,X_{r-1}=\gamma_k,\,
X_{r+1}=\delta_1,\ldots,X_{r+k}=\delta_k
\right)
\\
=
P\!\left(
X_\ell=\psi \mid
X_{\ell-k}=\gamma_1,\ldots,X_{\ell-1}=\gamma_k,\,
X_{\ell+1}=\delta_1,\ldots,X_{\ell+k}=\delta_k
\right).
\end{align*}
By the $k$-th order Markov property, conditioning on the full configuration may be
 reduced to conditioning on the $k$ nearest neighbors, 
\[
P(X_r=\psi \mid X_1,\ldots,X_{r-1},X_{r+1},\ldots,X_n)
=
P(X_r=\psi \mid X_{r-k},\ldots,X_{r-1},X_{r+1},\ldots,X_{r+k}),
\]
for $k+1\le r\le n-k$. Moreover, by homogeneity, this conditional distribution depends only on the local block
$(\gamma_1,\ldots,\gamma_k,\,\psi,\,\delta_1,\ldots,\delta_k)$ and is therefore independent
of the absolute position~$r$.
This is precisely the defining property of a homogeneous $k$-th order Markov chain.

Finally, clearing denominators yields the binomial
\[
p_{I,\,\gamma_1\ldots\gamma_k,\,\psi,\,\delta_1\ldots\delta_k,\,J}\,
p_{I',\,\gamma_1\ldots\gamma_k,\,\eta,\,\delta_1\ldots\delta_k,\,J'}
-p_{I,\,\gamma_1\ldots\gamma_k,\,\eta,\,\delta_1\ldots\delta_k,\,J}\,
p_{I',\,\gamma_1\ldots\gamma_k,\,\psi,\,\delta_1\ldots\delta_k,\,J'},
\]
which therefore belongs to $I(\mathcal{M})$.

For boundary cases (when $r\le k$ or $r>n-k$), the same cancellation argument applies:
the ratio is a product of the transition factors whose windows intersect $r$, and the set of
affected windows is again determined by the available neighbors.
Thus the ratio still depends only on the local configuration and not on the global position,
and the same cross-multiplication yields the stated binomial.
\end{proof}

\begin{example}
Consider the homogeneous illness--death model of Example~\ref{example:illness_death_model} with $n=4$ and $k=1$. Proposition~\ref{prop:homogeneous binomials} produces the binomial
\[
p_{0022} p_{0122} - p_{0012} p_{0222}
\]
by taking
\[
I=(0),\ I'=\emptyset,\ \Gamma=(0),\ \psi=2,\ \eta=1,\ \Delta=(2),\ J=\emptyset,\ J'=(2).
\]
Then, we have
\[
( I\Gamma\psi\Delta J,\ I'\Gamma\eta\Delta J',\
I\Gamma\eta\Delta J,\ I'\Gamma\psi\Delta J')
=(0022,0122,0012,0222),
\]
hence Proposition~\ref{prop:homogeneous binomials} yields that $p_{0022}p_{0122}-p_{0012}p_{0222}$ is in the ideal. This binomial does not appear among the generators of the corresponding nonhomogeneous ideal. Several further binomials appearing in the homogeneous ideal as computed in Example~\ref{example:illness_death_model} can be obtained in the same way. Table~\ref{tab:homogeneous-binomials} records the corresponding choices of the blocks $I$, $I'$, $\Gamma$, $\Delta$, $J$, and $J'$ for a selection of such binomials. The binomials in Table~\ref{tab:homogeneous-binomials} do not generate the entire vanishing ideal in Example~\ref{example:illness_death_model} in the homogeneous case. 
\begin{table}
\centering
\[
\begin{array}{c|c|c|c|c|c|c|c|c}
\text{Binomial} & I & I' & \Gamma & \psi & \eta & \Delta & J & J'\\
\hline
p_{1122}^2-p_{1112}p_{1222}
& (1) & \emptyset & (1) & 2 & 1 & (2) & \emptyset & (2)\\[1mm]
p_{0122}p_{1122}-p_{0112}p_{1222}
& (0) & \emptyset & (1) & 2 & 1 & (2) & \emptyset & (2)\\[1mm]
p_{0022}p_{0122}-p_{0012}p_{0222}
& (0) & \emptyset & (0) & 2 & 1 & (2) & \emptyset & (2)\\[1mm]
p_{1112}p_{0122}-p_{0112}p_{1122}
& (1) & (0) & (1) & 1 & 2 & (2) & \emptyset & \emptyset\\[1mm]
p_{0022}^2-p_{0002}p_{0222}
& (0) & \emptyset & (0) & 2 & 0 & (2) & \emptyset & (2)\\[1mm]
p_{0012}p_{0022}-p_{0002}p_{0122}
& (0) & \emptyset & (0) & 1 & 0 & (2) & \emptyset & (2)\\[1mm]
p_{0111}p_{0012}-p_{0011}p_{0112}
& \emptyset & \emptyset & (0) & 1 & 0 & (1) & (1) & (2)\\[1mm]
p_{0011}p_{0012}-p_{0001}p_{0112}
& (0) & \emptyset & (0) & 1 & 0 & (1) & \emptyset & (2)\\[1mm]
p_{0011}^2-p_{0001}p_{0111}
& (0) & \emptyset & (0) & 1 & 0 & (1) & \emptyset & (1)
\end{array}
\]
\caption{Examples of generators of the homogeneous ideal written in the notation of Proposition~\ref{prop:homogeneous binomials}.}
\label{tab:homogeneous-binomials}
\end{table}
\end{example}

\begin{remark}\label{rem:linear}
For a homogeneous $k$-th order multistate Markov model, the path probability
$p_{i_1\ldots i_n}$ depends only on the sufficient statistics, namely the counts
of length-$(k+1)$ contiguous blocks.
Consequently, if a permutation $\sigma$ of the time indices preserves these
block counts and the initial $k$-block, then
$p_{i_1\ldots i_n}=p_{\sigma(i_1)\ldots\sigma(i_n)}$
for all parameter values, and the linear binomial
$p_{i_1\ldots i_n}-p_{\sigma(i_1)\ldots\sigma(i_n)}$ lies in the vanishing ideal
$I(\mathcal{M})$.
\end{remark}

Computations indicate that the binomials in
Proposition~\ref{prop:homogeneous binomials} and Remark~\ref{rem:linear}
do not generate the vanishing ideal $I(\mathcal {V})=\ker(\phi^*)$ in the homogeneous case. This is illustrated in  Example~\ref{example:strict_inclusion_homogeneous}, where the Gr\"obner basis of $\ker(\phi^*)$ contains binomials of degree three that cannot be obtained from Proposition~\ref{prop:homogeneous binomials} and Remark~\ref{rem:linear}.
Moreover, the following example shows that the
inclusion of ideals in~\eqref{eqn:inclusion_of_inequalities}  can be strict.

\begin{example} \label{example:strict_inclusion_homogeneous} 
In this example we show that the inclusion of ideals~\eqref{eqn:inclusion_of_inequalities} 
\[
\sqrt{\ker (\phi^*)+\langle 1-\sum_{i_1,\ldots, i_n}p_{i_1\ldots i_n}}\rangle\subseteq I(\phi(\Theta)).
\]
can be strict in the homogeneous setting.
Consider the case $n=3$, $k=1$, with binary state space $S=\{0,1\}$.
Let
\[
I = \langle p_{i_1\ldots i_n} - 
\pi_{i_1\ldots i_k}\prod_{\ell=k+1}^n a_{i_{\ell-k}\ldots i_{\ell}}, \sum_{i_1\ldots i_k} \pi_{i_1\ldots i_k}-1, \sum_{i_{k+1}}a_{i_{1}\ldots i_k i_{k+1}}-1\rangle,
\]
be the ideal defining the homogeneous multistate Markov parametrization together with the normalization
constraints. Then 
 $I(\phi(\Theta)) = I \cap \mathbb{R}[p_{i_1\ldots i_n}]$.  Using \texttt{Macaulay2}, we find a minimal generating set 
 of this ideal:
 \begin{equation*}
 \begin{gathered}
 I(\phi(\Theta)) = \langle p_{000}+p_{001}+p_{010}+p_{011}+p_{100}+p_{101}+p_{110}+p_{111}-1,
 p_{011}p_{110}-p_{010}p_{111},\\ p_{011}p_{100}-p_{001}p_{110},
p_{011}p_{101}+p_{001}p_{110}-p_{010}p_{110}-p_{010}p_{111},\\
 p_{110}^2-p_{100}p_{111}-p_{101}p_{111}+p_{110}p_{111},\\
p_{010}p_{100}+p_{010}p_{101}+p_{100}p_{101}+p_{101}^2+p_{010}p_{110}+p_{101}p_{110}+p_{010}p_{111}+p_{101}p_{111}-p_{101},\\ p_{001}p_{100}+p_{001}p_{101}+p_{010}p_{101}+p_{100}p_{101}+p_{101}^2-\\-p_{001}p_{110}+p_{010}p_{110}+p_{101}p_{110}+p_{010}p_{111}+p_{101}p_{111}-p_{101},\\
p_{010}^2+2p_{010}p_{011}+p_{011}^2+p_{010}p_{101}-2p_{001}p_{110}+2p_{010}p_{110}-\\-p_{001}p_{111}+3p_{010}p_{111}+p_{011}p_{111}+p_{001}-p_{010}-p_{011}  
 \rangle.
 \end{gathered}
  \end{equation*}
On the other hand, a minimal generating set 
of $\sqrt{\ker(\phi^*)+\langle p_{000}+\cdots+p_{111}-1\rangle}$ is
\begin{equation*}
    \begin{gathered}
        \langle p_{000}+p_{001}+p_{010}+p_{011}+p_{100}+p_{101}+p_{110}+p_{111}-1,\\
        p_{011}p_{110}-p_{010}p_{111}, p_{011}p_{100}-p_{001}p_{110},\\ p_{001}p_{100}+p_{001}p_{101}+p_{010}p_{101}+p_{011}p_{101}+p_{100}p_{101}+p_{101}^2+p_{101}p_{110}+p_{101}p_{111}-p_{101}
     \rangle.
    \end{gathered}
\end{equation*}

A \texttt{Macaulay2} computation shows that 
\[
I(\phi(\Theta)) \neq \sqrt{\ker(\phi^*)+\langle p_{000}+\cdots+p_{111}-1\rangle}.
\]
Moreover, one can check that only the first three elements of the generators of $I(\phi(\Theta))$ belong to $\sqrt{\ker(\phi^*)+\langle p_{000}+\cdots+p_{111}-1\rangle}$.

Since both of these ideals are radical, by the ideal-variety correspondence, they define different sets over complex numbers. However, at this point we do not know if these sets are different also inside the probability simplex. We will show that this is the case in Section~\ref{subsec:homogeneous_MLE} by showing that maximum likelihood estimates over both sets are different.
\end{example}

The results of this section naturally lead to several open questions, which are listed in Section~\ref{sec:discussion}.

\section{Maximum likelihood estimation and likelihood geometry}
\label{sec:mle}

In this section, we compare classical statistical maximum likelihood estimation with algebraic approaches. We show
that for homogeneous multistate Markov models the algebraic formulation leads to more challenging questions than the
classical statistical approach. In Proposition~\ref{prop:model_is_not_equal_to_variety_interesected_with_probability_simplex} we show that, in general, $\mathcal{M} \neq \mathcal{V} \cap \Delta_{|S|^n-1}$ for homogeneous multistate Markov models. This result is obtained by comparing maximum likelihood estimates for the two models. 

\subsection{Nonhomogeneous multistate Markov model}

Let $\mathbf{X}^{(m)}=(X^{(m)}_1,\dots,X^{(m)}_{L+1})$, with $L>k$, for $m=1,\dots,M$, be $M$ independent realizations of a nonhomogeneous $k$-th order Markov chain with initial probabilities $\pi_{i_1\dots i_k} \geq 0$ and transition probabilities $a^{(\ell)}_{i_{\ell-k}\dots i_{\ell-1}i_\ell} \ge  0$.  Note that the forbidden transitions are accounted for through the conditions on the transition probabilities.  
For the $k$-th order nonhomogeneous model, the likelihood function of the parameters $(\pi,a)$ given the observed data $(\mathbf{X}^{(1)},\dots,\mathbf{X}^{(M)})$ is
\begin{align}
\label{eqn:likelihood_function_nonhomogeneous}
P(\mathbf{X}^{(1)},\dots,\mathbf{X}^{(M)})
&=
\prod_{m=1}^{M}
\left(
\pi_{X^{(m)}_1\dots X^{(m)}_k}
\prod_{\ell=k+1}^{L+1}
a^{(\ell)}_{X^{(m)}_{\ell-k}\dots X^{(m)}_{\ell-1}X^{(m)}_\ell}
\right).
\end{align}

The maximum likelihood estimators of the initial distribution and the transition probabilities are then the empirical proportions as given below.
\begin{align}
    \hat{\pi}_{i_1 \dots i_k} & = \frac{1}{M} \sum\limits_{m=1}^M  {\bf 1}( 
    X_{1}^{(m)} = i_1, \dots, X_k^{(m)} = i_k)  \nonumber \\
    \hat{a}_{i_{\ell-k}\dots i_{\ell-1} i_\ell}^{(\ell)} & = \frac{\sum\limits_{m=1}^M  {\bf 1}( 
    X_{\ell-k}^{(m)} = i_{\ell-k}, \dots, X_{\ell-1}^{(m)} = i_{\ell-1}, X_{\ell}^{(m)} = i_\ell)}{\sum\limits_{m=1}^M  {\bf 1}( 
    X_{\ell-k}^{(m)} = i_{\ell-k}, \dots, X_{\ell-1}^{(m)} = i_{\ell-1})}. \label{eqn:knhMCTP}
\end{align}
It is worth noting that, for a given history $(i_{\ell-k},\dots,i_{\ell-1})$, the above estimators determine
the transition probabilities to all but one state $i_\ell$. The remaining transition probability is 
obtained by enforcing the normalization constraints in~\eqref{constraints:korder}.

For a first-order nonhomogeneous Markov chain, these expressions simplify to
\begin{align}
\hat{a}^{(\ell)}_{i_{\ell-1} i_\ell}
&=
\frac{\sum_{m=1}^M \mathbf{1}\!\left(
X^{(m)}_{\ell-1}= i_{\ell-1},\;
X^{(m)}_{\ell}= i_\ell
\right)}
{\sum_{m=1}^M \mathbf{1}\!\left(
X^{(m)}_{\ell-1}= i_{\ell-1}
\right)}
=
\frac{\#\{\text{transitions } i_{\ell-1}\to i_\ell\}}
{\#\{\text{visits to state } i_{\ell-1}\}},
\label{eqn:1nhMCTP}
\end{align}
that is, the empirical conditional frequency of transitions from $i_{\ell-1}$ to $i_\ell$ at time $\ell$.

In this case, the maximum likelihood estimator of the initial distribution is
\begin{align}
\hat{\pi}_{i_1}
=
\frac{1}{M}\sum_{m=1}^M \mathbf{1}\!\left(X^{(m)}_1=i_1\right),
\qquad i_1\in S.
\label{eqn:1nhMCinitial}
\end{align}
The MLE of the joint probability of a length-$n$ path $p_{i_1\dots i_n}$ is obtained by substituting
$\hat{\pi}_{i_1}$ and $\hat{a}^{(\ell)}_{i_{\ell-1}i_\ell}$ into the factorization in~\eqref{eqn:pathprob}.
This estimator coincides with a special case of the MLE formula for decomposable hierarchical models
in algebraic statistics, expressed here in the language of Markov chains.

Let $\mathbf{u}=(u_{i_1\ldots i_n})$ denote the vector of observed counts  for joint state sequences,
where
\[
u_{i_1\ldots i_n}
=
\sum_{m=1}^M \mathbf{1}\!\left(X_1^{(m)}=i_1,\ldots,X_n^{(m)}=i_n\right)
\in \mathbb{N}
\]
is the number of times the sequence $(i_1,\ldots,i_n)$ appears in the data.
The total number of observations is
$\textstyle{M=\sum_{i_1\ldots i_n} u_{i_1\ldots i_n}.}$
For a fixed $k<n\le L+1$, the likelihood function of the length-$n$ path probabilities
based on the data $(X_1^{(m)},\ldots,X_n^{(m)})$, $m=1,\ldots,M$, is
\[
L(p;\mathbf{u})
=
\prod_{i_1,\ldots,i_n} p_{i_1\ldots i_n}^{\,u_{i_1\ldots i_n}}.
\]

Unlike the likelihood in~\eqref{eqn:likelihood_function_nonhomogeneous}, which uses the full
trajectory data, this likelihood depends only on the empirical counts of length-$n$ paths.
The goal of maximum likelihood estimation is to determine the probabilities
$p=(p_{i_1\ldots i_n})$ within the model that maximize $L(p;\mathbf{u})$.
This is a constrained optimization problem, where the constraints
are given by the polynomial equations defining the model, as described in
Section~\ref{sec:implicitization_nonhomogeneous}. In general, the likelihood equations may admit
multiple complex critical points.
For generic data vectors $\mathbf{u}$, the number of nonsingular complex critical points is constant and is
called the \emph{maximum likelihood degree} (ML degree) of the model
(see, e.g.,~\cite[Chapter~6]{drton-sturmfels-sullivant}).
The ML degree serves as a measure of the algebraic complexity of the maximum likelihood estimation problem, although one has to keep in mind that it does not account for critical points that are singular or on the boundary of the model.

It is a classical result that the ML degree of any decomposable hierarchical model is equal to one.
In particular, the maximum likelihood estimator exists uniquely and can be expressed as a rational
function of the data vector $\mathbf{u}$
(see, e.g.,~\cite[Section~2.1.7]{drton-sturmfels-sullivant}).
Since the general $k$-th order multistate Markov model is decomposable,
it follows that its maximum likelihood estimator admits an explicit closed-form expression.

To state this formula, we introduce notation for marginal counts of the data vector.
For a subset $R \subseteq \{1,\ldots,n\}$ and fixed states $(i_r)_{r\in R}$, define the marginal count
\[
u\big|_{\,i_r:\, r\in R}
\;=\;
\sum_{\substack{j_1,\ldots,j_n\\ j_r=i_r\ \text{for all } r\in R}}
u_{j_1,\ldots,j_n}.
\]

\begin{proposition}[\cite{drton-sturmfels-sullivant}, Proposition 2.1.7]\label{prop:mle hierarchical}
Consider the general $k$-order multistate Markov model and let $u$ be a vector of observed counts, and $M$ denote the total number of counts.
If all denominators below are nonzero, then the MLE is given by
    \[
    \hat{p}_{i_1\ldots i_n} = \frac{1}{M}\frac{\prod_{j=k+1}^n u|_{i_{j-k}\ldots i_j}}{\prod_{j=k+1}^{n-1} u|_{i_{j-k+1}\ldots i_j}},
    \]
    if the vector $u$ satisfies that none of the denominators vanish.
\end{proposition}
This expression coincides with the classical empirical transition–probability estimator,
written in joint-path coordinates.
Indeed, once the ML-estimators of the parameters are computed,
their product under the model parametrization yields exactly the normalized expressions
in Proposition~\ref{prop:mle hierarchical}.

As an illustration, consider the case $k=1$.
The relevant marginal counts are
\[
u\big|_{i_\ell}
=
\#\{\text{visits to state } i_\ell\},
\qquad
u\big|_{i_{\ell-1}i_\ell}
=
\#\{i_{\ell-1}\to i_\ell \text{ direct transitions}\}.
\]
It follows that
\begin{align*}
\hat{p}_{i_1\ldots i_n}
&=
\frac{1}{M}\,
\frac{\prod_{\ell=2}^n
\#\{i_{\ell-1}\to i_\ell \text{ direct transitions}\}}
{\prod_{\ell=2}^{n-1}
\#\{\text{visits to state } i_\ell\}}
\\
&=
\frac{\#\{i_1\to i_2 \text{ direct transitions}\}}{M}
\prod_{\ell=3}^n
\frac{\#\{i_{\ell-1}\to i_\ell \text{ direct transitions}\}}
{\#\{\text{visits to state } i_{\ell-1}\}}
\\
&=
\frac{\#\{\text{visits to state } i_1\}}{M}\,
\frac{\#\{i_1\to i_2 \text{ direct transitions}\}}
{\#\{\text{visits to state } i_1\}}
\prod_{\ell=3}^n
\frac{\#\{i_{\ell-1}\to i_\ell \text{ direct transitions}\}}
{\#\{\text{visits to state } i_{\ell-1} \}}
\\
&=
\hat{\pi}_{i_1}\prod_{\ell=2}^n \hat{a}^{(\ell)}_{i_{\ell-1} i_\ell}.
\end{align*}

\paragraph{Forbidden transitions.}
We now turn to models in which certain transitions are forbidden.
For such models, the conclusion of Proposition~\ref{prop:mle hierarchical} continues to hold
for all nonvanishing path probabilities.
Indeed, if a transition is forbidden, then any path containing that transition has probability
zero under the model, and the corresponding empirical counts
$u_{i_1\ldots i_n}$ must also be zero.
Consequently, these paths do not contribute to the likelihood function.
From an algebraic perspective, the vanishing ideal of a model with forbidden transitions
is obtained from that of the general model by adding the linear equations
$p_{i_1\ldots i_n}=0$ for all forbidden paths.

The maximum likelihood estimation problem in this setting can therefore be formulated as
the maximization of the likelihood function over the general model, subject to additional
polynomial constraints.
A direct computation using Lagrange multipliers shows that the estimator given in
Proposition~\ref{prop:mle hierarchical} satisfies the likelihood equations for all
coordinates corresponding to nonvanishing probabilities.
Thus, the same closed-form expression yields the maximum likelihood estimator on the
restricted model, after removing the zero-probability paths.

\subsection{Homogeneous multistate Markov model} \label{subsec:homogeneous_MLE}

In the $k$-th order homogeneous Markov model, we assume that the same transition
probabilities apply at every $k$-step transition, independently of the time index.
As a consequence, the likelihood function has the same form as in
\eqref{eqn:likelihood_function_nonhomogeneous}, except that the transition
probabilities no longer depend on $\ell$:
\begin{align}
P(\mathbf{X}^{(1)}, \dots, \mathbf{X}^{(M)})
=
\prod_{m=1}^{M}
\biggl(
\pi_{X_1^{(m)} \dots X_k^{(m)}}
\prod_{\ell=k+1}^{L+1}
a_{X_{\ell-k}^{(m)} \dots X_{\ell-1}^{(m)} X_\ell^{(m)}}
\biggr).
\end{align}

The maximum likelihood estimator of the initial distribution remains the same as
in~\eqref{eqn:knhMCTP}. In contrast, the estimator of the transition probabilities
in this case are obtained by summing over all time points~$\ell$:
\begin{align}
\hat{a}_{i_{\ell-k} \dots i_{\ell-1} i_\ell}
=
\frac{
\sum\limits_{m=1}^M \sum\limits_{\ell=k+1}^{L+1}
{\bf 1}\!\left(
X_{\ell-k}^{(m)} = i_{\ell-k}, \dots,
X_{\ell-1}^{(m)} = i_{\ell-1},
X_{\ell}^{(m)} = i_{\ell}
\right)
}{
\sum\limits_{m=1}^M \sum\limits_{\ell=k+1}^{L+1}
{\bf 1}\!\left(
X_{\ell-k}^{(m)} = i_{\ell-k}, \dots,
X_{\ell-1}^{(m)} = i_{\ell-1}
\right)
}.
\label{eqn:khMCTP}
\end{align}

The variances of the maximum likelihood estimators, under both nonhomogeneous and homogeneous multistate Markov models, can be obtained from the Fisher information matrix in the usual way; see, for example, \cite[Chapter~7]{vanderVaart2000}.

From the algebraic perspective, a natural first idea is to apply existing approaches for computing the MLE in log-linear models since $I(\mathcal{V})$ is a toric ideal. However, in Example~\ref{example:strict_inclusion_homogeneous} we showed that $I(\mathcal{V}) + \left\langle 1-\sum_{i_1,\ldots, i_n}p_{i_1\ldots i_n} \right\rangle$ and $I(\mathcal{M})$ can be different ideals. In Example~\ref{example:homogeneous_MLE} we demonstrate on an example that one indeed gets different answers using formulas~\eqref{eqn:knhMCTP}  and~\eqref{eqn:khMCTP} and an approach for computing MLE for log-linear models.

For general statistical models defined by polynomial equations and inequalities, one can find the MLE by constructing the \emph{likelihood ideal} (\cite[Algorithm 7.2.4]{sullivant2023algebraic}). The likelihood ideal is generated by polynomials that are obtained by the method of Lagrange multipliers for the constrained optimization problem of maximizing the likelihood function subject to the polynomial equations defining the model. In certain cases, MLE is the nonnegative real solution of the likelihood ideal that maximizes the likelihood function. However, there are also some caveats such as that the MLE can be a singular or boundary point of the model in which case it is usually not the solution of the likelihood ideal. Moreover, as we will discuss in Example~\ref{example:homogeneous_MLE}, there can be practical challenges working with the likelihood ideal.  Although algebraic techniques have frequently played an important role in developing deeper insight into maximum likelihood estimation~\cite{kubjas2015fixed}, for homogeneous multistate Markov models a direct statistical approach using standard formulas provides a simpler solution, whereas the algebraic approach is unnecessarily complicated.

\begin{example} \label{example:homogeneous_MLE}
We continue Example~\ref{example:illness_death_model} for the illness-death model (Figure~\ref{fig:illness}) in the homogeneous case with observations at four (discrete) time points. We consider a vector of counts, denoted by $\textbf{u}$, which we  obtain by simulating according to the data-generating mechanism from a homogeneous illness-death Markov model. Precisely, this means that we iteratively simulate $M=1000$ individual  trajectories by first sampling $M$ initial states according to $\boldsymbol{\pi}_{\text{gen}}$, and subsequently sampling the next three states in a sequential manner according to the transition probability matrix $\textbf{A}_{\text{gen}}$. This results in $M$ four-time-point trajectories and we count how many times each specific path occurs according to the 14 possible path realizations seen in Table~\ref{tab:obsevations4}, storing the frequency counts in vector $\textbf{u}$. The $\boldsymbol{\pi}_{\text{gen}}$ and $\textbf{A}_{\text{gen}}$ we used are:
\[
\boldsymbol{\pi}_{\text{gen}} = \begin{pmatrix}
    0.700\\
    0.300\\
    0
\end{pmatrix} ,\ 
\textbf{A}_{\text{gen}} = \begin{pmatrix}
    0.600 &0.275& 0.125\\
     0  & 0.750 & 0.250\\
     0 & 0 & 1
\end{pmatrix}.
\]
This gives rise to the vector $\textbf{u}= (140,59,25,82,30,61,114,48,54,92,137,53,60,45)$, with the entries ordered in the same way as in Example~\ref{example:illness_death_model}.  The vector of the empirical probabilities  is 
\[
\frac{\textbf{u}}{|\textbf{u}|}=
(0.140,\,0.059,\,0.025,\,0.082,\,0.030,\,0.061,\,0.114,\,0.048,\,0.054,\,0.092,\,0.137,\,0.053,\,0.060,\,0.045).
\]
This vector does not belong to the model as can be checked by evaluating polynomials from Example~\ref{example:illness_death_model} on it. For example, the first generator $p_{1 1 2 2}^2-p_{1 1 1 2}p_{1 2 2 2}$ in the ideal evaluates to $0.060^2-0.053 \cdot 0.045 = 0.001215$.

We can now obtain the estimates of the initial distribution using equation ~\eqref{eqn:knhMCTP} and of the transition probabilities using equation ~\eqref{eqn:khMCTP}:

\[
\hat{\boldsymbol{\pi}} = \begin{pmatrix}
    0.705\\
    0.295\\
    0
\end{pmatrix} ,\
\hat{\textbf{A}} = \begin{pmatrix}
    0.574 &0.292& 0.134\\
     0  & 0.763 & 0.237\\
     0 & 0 & 1
\end{pmatrix}.
\]
By substituting the estimates to the parametrization map $\phi$ we get the maximum likelihood estimates of the path probabilities
\[
\hat{\textbf{p}} = (0.133,\, 0.068,\, 0.031,\, 0.090,\, 0.028,\, 0.054,\, 0.120,\, 0.037,\, 0.049,\, 0.049,\, 0.131,\, 0.041,\,  0.053,\, 0.070).
\]
The difference $({\textbf{u}}/{|\textbf{u}|}) - \hat{\textbf{p}}$ is
\[
(0.007,\,-0.009,\,-0.006,\,-0.008,\,0.002,\,0.007,\,-0.006,\,0.011,\,0.005,\,0.043,\,0.006,\,0.012,\,0.007,\,-0.025).
\]
Observe that the empirical proportions of paths (and also the counts $\textbf{u}$) do not match with the MLE.  This is because the paths are governed by the multistate model, with the MLE obtained under this model by considering the initial distribution and the transition probability matrix, whereas the simple empirical proportions are obtained without imposing any particular model structure.

We now turn to algebraic approaches. 
Let $\textbf{A} \in \mathbb{Z}^{d \times n}$ be a matrix,
 $I_{\textbf{A}}$ the associated toric ideal, 
$\textbf{u} \in \mathbb{N}^n$ a vector of observed counts and $M$ the total number of observations.
By~\cite[Corollary 7.3.9]{sullivant2023algebraic}, the maximum likelihood estimate  $\hat{\textbf{p}}'$ on $V(I_{\textbf{A}}) \cap \Delta_{|S|^n-1}$ is the
unique nonnegative solution of the system
\begin{equation} \label{theorem_birch}
M \textbf{A} \hat{\textbf{p}}' = \textbf{A} \textbf{u},
\qquad
\hat{\textbf{p}}' \in V(I_{\textbf{A}}),
\end{equation}
whenever such a solution exists.

In our example, the ideal corresponding to equations~\eqref{theorem_birch} has dimension $0$ and degree $17$, and therefore has
$17$ complex solutions, counted with multiplicity. There are three real solutions and exactly one solution is nonnegative:
\[
\hat{\textbf{p}}' = (0.14,\, 0.053,\, 0.034,\, 0.081,\, 0.031,\, 0.055,\, 0.124,\, 0.048,\, 0.05,\, 0.089,\, 0.134,\, 0.051,\, 0.054,\, 0.056).
\]
Note that $\hat{\textbf{p}}\neq \hat{\textbf{p}}'$. 
This implies that $\mathcal{M}$ and $\textstyle{V(I_{\textbf{A}}) \cap \Delta_{|S|^n-1}=\mathcal{V} \cap \Delta_{|S|^n-1}}$ are not equal as the maximum likelihood estimates over these sets are different. We also have a strict inclusion of the ideals  $\textstyle{\sqrt{I_{\textbf{A}} + \langle 1-\sum_{i_1,\ldots, i_4}p_{i_1\ldots i_4}}\rangle \subset I(\mathcal{M})}$.  The ideal $I_{\textbf{A}}$ is a toric ideal of dimension seven and degree $31$ minimally generated by $18$ binomials. The ideal $\textstyle{\sqrt{I_{\textbf{A}} + \langle 1-\sum_{i_1,\ldots, i_4}p_{i_1\ldots i_4}}\rangle}$ is an ideal of dimension six and degree $31$ that is minimally generated by $19$ polynomials. On the other hand, the vanishing ideal of the model $I(\mathcal{M})$ is a radical ideal of dimension four and degree $25$ minimally generated by $31$ polynomials. 

The second algebraic approach is constructing the likelihood ideal for $I(\mathcal{M})$. A step towards this is constructing the augmented Jacobian of $I(\mathcal{M})$, a modification of the Jacobian matrix. In this particular example, the augmented Jacobian is a $31 \times 14$ matrix and at a critical point $p$ of the likelihood function, the vector $\textbf{u}$ lies in the row span of the augmented Jacobian. Due to the computational burden, we were not able to study the likelihood ideal, however, we could check that at the optimal $\hat{p}$, $\textbf{u}$ lies indeed in the row span of the augmented Jacobian. Moreover, $\hat{\textbf{p}}$ is a regular point of the interior of the model.

This example shows that
classical statistical methods provide a simpler and more direct route to the MLE in this case although the MLE is a regular point of the interior of the model.
\end{example}

\begin{proposition} \label{prop:model_is_not_equal_to_variety_interesected_with_probability_simplex}
In general, $\mathcal{M} \neq \mathcal{V} \cap \Delta_{|S|^n-1}$ for homogeneous multistate Markov models.
\end{proposition}

\begin{proof}
In Example~\ref{example:homogeneous_MLE}, the MLE over $\mathcal{M}$ is $\hat{\textbf{p}}$ and the MLE over $\mathcal{V} \cap \Delta_{|S|^n-1}$ is $\hat{\textbf{p}}'$. If the two sets were equal, then the MLEs over the two sets would also be equal.
\end{proof}

\subsection{Identifiability}

A model is called identifiable if its parameters can be uniquely recovered from
the observed probability distribution; see, for example,
\cite[Chapter~16]{sullivant2023algebraic} for an introduction.

\begin{definition}
\label{def::TypesofID}
Let $\phi \colon \Theta \rightarrow \mathcal{M}_{\phi}$ be a rational map from a
finite-dimensional parameter space $\Theta \subseteq \mathbb{R}^k$ to a
statistical model $\mathcal{M}_{\phi}$, with $\mathrm{im}(\phi)=\mathcal{M}_{\phi}$.
The model is said to be \emph{globally identifiable} if
\[
\phi^{-1}(\phi(\theta))=\theta
\quad\text{for all }\theta\in\Theta.
\]
\end{definition}
If we take $M=1$ in the previous two subsections, then the formulas
in \eqref{eqn:knhMCTP} recover the parameters of the nonhomogeneous model and~\eqref{eqn:knhMCTP} and~\eqref{eqn:khMCTP} recover the parameters of the homogeneous model uniquely from
the joint path probabilities.
This holds both for general $k$-th order multistate Markov models and for
models with forbidden transitions.
Therefore, all these models are globally identifiable.

\section{Examples and applications }\label{sec:applications} 

For the purposes of a computational illustration, we use data  which includes all of the English words that have ever been used in William Shakespeare’s works. We use this dataset to construct a general  discrete-time multistate Markov  model. The state space $S$ includes 27 states: 26 states representing the 26 letters in the English alphabet, and one absorbing state that represents a space (hereafter denoted by $\mathcal\Box$).
 Ultimately, our construction entails that each word becomes a specific path and gets absorbed into $\mathcal\Box$. 

Words in the corpus evidently have varying lengths. In order to ensure that each trajectory is of the same length, and to thus use a common time horizon across all trajectories, we consider a maximum length $L$ corresponding to the (second) longest word in the state space (note that the longest word \textit{honorificabilitudinitatibus} from medieval Latin, with 27 characters, was disregarded for the sake of convenience). All of the words shorter than this were then padded with a space character (absorbing state) up until the end of the time horizon. Therefore, we may write the time horizon of the process as $\ell = 1, 2, \dots, L, L+1$.  Including the $L+1$ time point ensures that we can explicitly capture the final letter-to-space transition. To summarize, we represent words as Markovian paths, and a word of length $k < L$ is modeled as a trajectory $X_1, \dots, X_{L+1}$, whereby $X_\ell$ is the letter at the $\ell$-th time point and $X_t = \mathcal\Box$ for times $\ell\geq k+1$.  We are also then interested in the number of times that each Markov trajectory, or word, appears in the dataset, and we therefore explicitly keep a count of the number of occurrences of each word since this is essential for the construction of a likelihood function. As way of illustration, in our dataset, we have $L=17$, so if we consider the word ``and", then $X_1 = \text{a}$, $X_2 = \text{n}$, and $X_3=\text{d}$; $X_4 = X_5 = \dots = X_{17} = X_{18} = \mathcal\Box$. We finally also remark that words consisting of just one letter, such as ``s" (resulting from the use of an apostrophe to indicate possession or prefixes) have been discarded from the dataset. 

\subsection{Statistical computations for the MLE}\label{subsec:mle_stat_comp} 

Under the case of  the first-order homogeneous Markov chain, it is important to recall from Section~\ref{sec:Preliminaries} that we deal with just one transition probability matrix $\mathbf{A} =(a_{ij})$ 
of dimension ${27\times27}$, a square matrix of the same order as the dimension of the state space $S$. By contrast, in the first-order nonhomogeneous case, we emphasize that we have $L-1$ such matrices; indeed, one per each adjacent time step pair. To fully construct a likelihood function under both of these cases, we will need the initial distribution (row) vector denoted as $\boldsymbol{\pi} = (\pi_{i}), i \in S$, where $\pi_{i}=P(X_1=i)$, which provides us with the probability that a word  begins with a certain character $i$. We also need $\mathbf{A}$ to construct the likelihood. Nonetheless, we note that in both the homogeneous and nonhomogeneous cases, a single common initial distribution $\boldsymbol{\pi}$ governs the distribution of the first character. The model is nonhomogeneous when the transition probabilities depend on the time step.

Directly using equation~\eqref{eqn:khMCTP} with $k=1$, we can compute the MLE's of the $a_{ij}$ entries of $\mathbf{A}$. The intuition, applied to this dataset, is that they represent a ratio defined by the following quantities: in the numerator, we count how many times, across all words and all time instances, the process was at letter $i$ by time $\ell-1$ and immediately was at letter $j$ by the subsequent time step $\ell$; in the denominator, we count how many times the process was at letter $i$ and observed a next transition to any of the other states. 

We perform the statistical computation with the $27$ states using \texttt{R} language, and we report an illustrative submatrix denoted $\hat{\mathbf{A}}_{\text{sub}}$ from the overall $\hat{\mathbf{A}}$ which contains common two-letter consecutive sequences (called bigrams) in the English language \footnote{\url{https://mathcenter.oxford.emory.edu/site/math125/englishLetterFreqs/}\label{website_bigrams}}: 

\[ \hat{\mathbf{A}}_{\text{sub}}=
\left(
\begin{array}{c|ccccc}
      & h     & e     & n     & s     & t \\
\hline
t & 0.347 & 0.074 & 0.001 & 0.016 & 0.013 \\
h & 0.000 & 0.356 & 0.003 & 0.003 & 0.027 \\
i & 0.000 & 0.045 & 0.236 & 0.156 & 0.129 \\
e & 0.002 & 0.039 & 0.089 & 0.075 & 0.035 \\
n & 0.001 & 0.086 & 0.009 & 0.033 & 0.089
\end{array}
\right)
\]

From the above submatrix, it is of interest to remark that the two-letter consecutive sequence ``th" has a relatively high probability of occurring (it is the highest probability in this submatrix) and this is due to the existence of common words and articles that appear many times in English such as ``the", ``this", ``that", etc. Precisely, ``th" is the most common bigram by frequency. The second most common bigram is ``he" and our submatrix also demonstrates this\textsuperscript{\ref{website_bigrams}}. Additionally, some transition pairs like $h\rightarrow h$ occur with probability $0$, meaning there exists no such bigram in the Shakeperean English (Elizabethan) corpus. This may slightly differ to modern English, where a word such as ``bathhouse" does actually have that precise two-letter sequence. Lastly, with respect to the MLE for the initial distribution $\boldsymbol{\pi}$, we obtain a couple of values such as: $\hat{\pi}_a =0.081$, $\hat{\pi}_t=0.147$, and $\hat{\pi}_s = 0.079$. The interpretation, as way of illustration, is that roughly $15\%$ of words in the Shakespearan English corpus start with the letter ``t".  As expected, the values for $\hat{\pi}_x$ or $\hat{\pi}_z$ are very small.

It is furthermore of interest, in relation to what has been discussed previously, to visualize an example of joint path probabilities using this dataset. Taking the word ``the" as illustration, we can compute its probability of occurrence as such: 
\begin{align}
    P(X_1=t)P(X_2=h |X_1=t)P(X_3=e|X_2=h).
\end{align}
\label{eqn: applic_path_prob}
This product of conditional probabilities can be directly computed using the relevant entries of $\hat{\mathbf{A}}_{\text{sub}}$, and by noticing that the first factor is actually just the initial distribution of the character $t$, which we expressed as $\hat{\pi}_t=0.147$. Therefore, the probability of the word ``the" is equal to the product $0.147 \times 0.347 \times 0.356 = 0.0176.$ 

In the first-order nonhomogeneous case, we present the MLE's for position-specific transition probabilities for some selected letter pairs in Table~\ref{tab:mle_nonhom}.

\begin{table}[h]
\centering
\caption{Selected MLE's for some first-order letter transition pairs at the $16$ time steps being considered}
\label{tab:mle_nonhom}
\begin{tabular}{c ccccc}
\toprule
$\ell$ & $t\!\to\! h$ & $h\!\to\! e$ & $i\!\to\! n$ & $e\!\to\! r$ & $a\!\to\! l$ \\
\midrule
2  & 0.6665 & 0.3522 & 0.3803 & 0.0698 & 0.0899 \\
3  & 0.0680 & 0.4519 & 0.1205 & 0.1349 & 0.0737 \\
4  & 0.2773 & 0.5339 & 0.2151 & 0.0783 & 0.0857 \\
5  & 0.1346 & 0.1204 & 0.2482 & 0.1243 & 0.0910 \\
6  & 0.0578 & 0.1951 & 0.3490 & 0.1934 & 0.0587 \\
7  & 0.0420 & 0.1435 & 0.3281 & 0.1747 & 0.1075 \\
8  & 0.0344 & 0.0405 & 0.3372 & 0.1691 & 0.0892 \\
9  & 0.0524 & 0.0577 & 0.2814 & 0.1309 & 0.0578 \\
10  & 0.0489 & 0.1706 & 0.1717 & 0.1752 & 0.0633 \\
11 & 0.0252 & 0.0686 & 0.1987 & 0.0930 & 0.1293 \\
12 & 0.0693 & 0.0161 & 0.2118 & 0.0420 & 0.2667 \\
13 & 0.0132 & 0.0000 & 0.1579 & 0.0243 & 0.0222 \\
14 & 0.0000 & 0.0000 & 0.4630 & 0.0000 & 0.6000 \\
15 & 0.0000 & 0.0000 & 0.0000 & 0.0000 & 0.0000 \\
16 & 0.0000 & 0.0000 & 0.0000 & 0.0000 & 0.0000 \\
17 & 0.0000 & 0.0000 & 0.0000 & 0.0000 & 0.0000 \\
\bottomrule
\end{tabular}
\end{table}

As with the discussion concerning the homogeneous case, we directly observe that the $t\rightarrow h$ transition occurs with high probability, and here specifically we can discern that it occurs at the very first two adjacent time pairs ($\ell=2$ in Table~\ref{tab:mle_nonhom}). This is again due to the common word instances that start with this bigram in the English language. In that sense, the nonhomogeneous setting provides us with more information because we can identify specific positions within words that carry a high likelihood of encountering a specific bigram; something we were not able to comment on in the homogeneous case. Similarly, the pair $i \rightarrow n$ occurs with high probability both at the initial adjacent time instances, but also even later in positions $13\rightarrow 14$ ($\ell=14$ in Table~\ref{tab:mle_nonhom}) for longer words. 

We can perform the same calculation for the probability of the word ``the" using equation~\eqref{eqn: applic_path_prob}, but with the understanding that these are now time-dependent conditional probabilities (the estimators of transition probabilities from Table~\ref{tab:mle_nonhom} specifically depend on the character placements within the actual word). Observe that we employ the same initial distribution, as was mentioned previously. We obtain the product: $0.147 \times 0.6665 \times 0.4519 = 0.0443. $ Under the nonhomogeneous case, the probability of encountering the word ``the" is higher than in the homogeneous case, but this is unsurprising due to our previous discussion. As an interesting point of comparison, the empirical proportion of the word ``the" in the cleaned Shakespeare corpus used for the analysis is $0.0325.$ This is closer to the estimate of the probability of path ``the" under the nonhomogeneous case ($0.0443$) than the homogeneous case ($0.0176$).

\begin{remark} 
In light of the discussion in Example~\ref{example:homogeneous_MLE}, we want to investigate in the dataset the possible discrepancies between the ordering of observed word counts and the ordering of the fitted MLE probabilities. To do this, we compare the ranks induced by the empirical counts with the ranks induced by these fitted probabilities. 

Using the Shakespeare corpus dataset, we have a column that identifies the number of occurrences of a particular word trajectory. We can thus easily calculate the empirical proportion, as well as calculate its fitted probability under the homogeneous Markov model in the same manner as in equation~\eqref{eqn: applic_path_prob}. We can subsequently rank the words according to these two computations. We note that words having large differences between these two ranks are those whose empirical prominence in the data differs the most from their fitted probability under the model. 

\begin{table}[ht]
\centering
\caption{Subset of words with high discrepancies between observed counts and fitted MLEs}
\label{tab:word_rank_discrepancies}
\begin{tabular}{l l l l l}
\toprule
Direction & Word & Observed Count & Empirical Proportion & Fitted MLE \\
\midrule
Higher in observations & \textit{northumberland} & 163 & \(1.764\times 10^{-4}\) & \(3.840\times 10^{-17}\)  \\
Higher in observations & \textit{rosencrantz} & 78 & \(8.440\times 10^{-5}\) & \(2.301\times 10^{-17}\)  \\
Higher in observations & \textit{cleopatra} & 277 & \(2.997\times 10^{-4}\) & \(6.797\times 10^{-15}\)  \\
\midrule
Higher under MLE & \textit{bus} & 5 & \(5.410\times 10^{-6}\) & \(4.652\times 10^{-4}\)  \\
Higher under MLE & \textit{ace} & 5 & \(5.410\times 10^{-6}\) & \(1.997\times 10^{-4}\)  \\
\bottomrule
\end{tabular}
\end{table}

As way of illustration, in Table~\ref{tab:word_rank_discrepancies}, we see that words such as ``northumberland" or ``cleopatra" often show up because they appear repeatedly in Shakespeare's works. Nonetheless, they are assigned very small fitted probabilities because our homogeneous model captures general transition patterns between letters, instead of play-specific repetitions of particular names. 

As in Example~\ref{example:homogeneous_MLE}, the vector of empirical proportions of words does not belong to the variety of the vanishing ideal of the first-order multistate Markov model. Consider four paths (``trace", ``grade", ``trade", ``grace") with respective counts $(12, 0, 43, 622)$. By Proposition~\ref{prop:generating set general} ($I = (\text{tr}), I' = (\text{gr}), J = (\text{a}), S = (\text{ce}), S' = (\text{de})$), the polynomial
\[
p_{\text{trace}}p_{\text{grade}}-p_{\text{trade}}p_{\text{grace}}
\]
belongs to the vanishing ideal of the nonhomogeneous first-order multistate Markov model, and hence also to that of the homogeneous first-order multistate Markov model.  
The first term of the polynomial evaluates to zero, but the second does not. Hence, the polynomial vanishes neither on the counts nor on the empirical proportions. Such examples are easy to construct: for any two words sharing a letter at the same position, if swapping all letters after that position yields at least one letter combination that is not an actual word, its observed count will be zero, making one term of the polynomial zero while the other remains nonzero.
\end{remark}

\textbf{Reversible illness-death model.} To continue, one final application of this dataset is to collapse the state space $S$ into just three states consisting of: vowels, consonants, and the absorbing state of a space ($\mathcal\Box$). That way, we may introduce, as discussed in Section~\ref{sec:Intro}, the analogue of an illness-death model, but with back transitions allowed between states $0$ and $1$. This is an  application  of a reversible illness-death model introduced in Example~\ref{example:revillness_death_model}. For clarity, we may designate that all consonants (denoted as $C$ thereafter) are the ``healthy" state ($0$), the vowels $V$ are the ``diseased" state ($1$), and the space character $\mathcal\Box$ is the absorbing ``death" state ($2$). This means that, overall there are six permissible transitions at each time step: $V\rightarrow V$, $V\rightarrow C$, $V\rightarrow \mathcal\Box$, $C\rightarrow C$, $C\rightarrow V$, $C \rightarrow \mathcal{\Box}.$ For convenience sake, we may refer to this new reduced 3-state state space as $S^*$ so that $S^{*}=\{V,C,\mathcal{\Box}$\}.

We now use this example to demonstrate the construction of a homogeneous $2$-nd order  Markov multistate model ($k=2$) and also perform similar work as above to obtain the MLE's of the entries of the transition probability matrix. Note that these entries are probabilities which are conditioned on two backward time steps, such as: $a_{i_{\ell-2} i_{\ell-1} i_\ell}= P(X_\ell = i_\ell | X_{\ell-2}=i_{\ell-2},X_{\ell-1}=i_{\ell-1})$ with $i_{\ell-2}, i_{\ell-1}, i_{\ell} \in S^{*}$, and $(i_{\ell-2}, i_{\ell-1})\in \{V,C\}^{2}$. This means that the two previous time steps are combinations of $(V,C)$ pairs, and we want to obtain the conditional transition probability of the next time step being one of $(V,C,\mathcal\Box).$ 

We show the corresponding estimators of the entries of the transition probability matrix $A_2 = (a_{i_{\ell-2} i_{\ell-1} i_\ell})$ of a $2$-nd order homogeneous Markov multistate model according to equation \eqref{eqn:khMCTP}. Due to rounding, not all rows sum exactly to $1$. 
\[ {\boldsymbol{  \widehat A}_2}
=
\left(
\begin{array}{c|ccc}
      & V & C & \mathcal\Box \\
\hline
VV & 0.02 & 0.77 & 0.20 \\
VC & 0.23 & 0.40 & 0.37 \\
CV & 0.16 & 0.64 & 0.20 \\
CC & 0.49 & 0.11 & 0.39
\end{array}
\right).
\]
The above matrix offers evident conclusions about what we already know concerning the English language, such as the fact that the likelihood of encountering words with three consecutive vowels or three consecutive consonants is low. Thus, with relatively high probability, when the two previous letters are vowels, then we can expect a transition into a consonant for the third character. Interestingly, we can apply this understanding to reversible illness-death model dynamics. Indeed, we can conceptualize that the transition probabilities exhibit a dependence on the duration (number of time steps) in the current state, where for example, a prolonged stay in either the ``healthy" or ``diseased" state may influence the stochastic movement between said states. This statistical premise is known as a semi-Markov model, and arises quite naturally when discussing multistate models in general. This effect cannot be seen when only discussing first-order models, hence why higher order models can provide valuable insights about the state occupancy time.

Finally, for this example, we specify an initial distribution row vector $\boldsymbol{\pi}_{2} = (\pi_{i,j})_{(i,j)\in\{V,C\}^{2}}$. It specifies how the second-order homogeneous model is initially configured by looking at the vowel-consonant pair jointly at times $l=1$ and $l=2$. Hence, we may write $\pi_{ij}=P(X_{1}=i,X_{2}=j).$ Using the estimator of the initial distribution given in equation \eqref{eqn:knhMCTP}, we obtain the following MLE of the initial distribution (again noting that it does not sum to $1$ due to rounding): 
\begin{align*}
    \boldsymbol{\hat{\pi}}^{(2)}  &= P\big((X_1=V, X_{2}=V), (X_{1}=V,X_{2}=C),(X_{1}=C,X_{2}=V),(X_{1}=C,X_{2}=C)\big) = \\
    &= (0.01,0.19,0.55,0.24).
\end{align*}
We also remark, as is standard practice in illness-death modelling, that we have assumed that the initial distribution of the system is only assigned to states $0$ and $1$ ($C$ and $V$ respectively for the example). This means that the process is assumed to start in a nonabsorbing state, and at the first two time instances, subjects are either healthy $(C)$ or with disease $(V)$, but not  dead $(\mathcal{\Box}).$ We may lastly also use $\boldsymbol{  \widehat A}_2$ together with the initial distribution to obtain all possible combinations of three-step path probabilities. Table~\ref{tab:path_probs_secondorder} depicts the results. We see that the path $CVC$ arises with the highest probability, which again reflects the construction of many words in English. 

\begin{table}[ht]
\centering
\caption{ML-estimates of all path probabilities under the second-order nonhomogeneous model}
\label{tab:path_probs_secondorder}
\begin{tabular}{c c c c}
\toprule
path & $p_{\text{path}}$ & Computation & ML-estimates \\
\midrule
VVV & $P(VV)P(V| VV)$ & $0.01 \times 0.02$ & 0.0002 \\ 
VVC & $P(VV)P(C|VV)$ & $0.01\times0.77$  & 0.0077 \\
VV$\mathcal\Box$ & $P(VV)P(\mathcal\Box|VV)$& $0.01\times0.20$ & 0.0020    \\
VCV & $P(VC)P(V|VC)$  & $0.19\times0.23$ & 0.0437  \\
VCC & $P(VC)P(C|VC)$  & $0.19\times0.40$ & 0.0760 \\
VC$\mathcal\Box$ & $P(VC)P(\mathcal\Box|VC)$ &$0.19\times0.37$ & 0.0703   \\
CVV & $P(CV)P(V|CV)$ & $0.55\times 0.16$  & 0.0880  \\
CVC & $P(CV)P(C|CV)$ & $0.55 \times 0.64$ & 0.3520 \\
CV$\mathcal\Box$ & $P(CV)P(\mathcal\Box|CV)$ & $0.55\times 0.20$ & 0.1100  \\
CCV & $P(CC)P(V|CC)$ & $0.24\times 0.49$  & 0.1176 \\
CCC & $P(CC)P(C|CC)$  & $0.24 \times 0.11$ &  0.0264\\
CC$\mathcal\Box$ & $P(CC)P(\mathcal\Box|CC)$ & $0.24 \times 0.39$ & 0.0936 \\
\bottomrule
\end{tabular}
\end{table}

\section{Discussion and future directions}
\label{sec:discussion}

In this paper, we studied discrete-time multistate Markov models from the perspective of
algebraic statistics and made explicit the connection between the multistate modelling
framework and algebraic–geometric descriptions of statistical models.
To the best of our knowledge, this is the first work that systematically bridges these two
areas and establishes a precise correspondence between them.
While algebraic statistics is a comparatively recent field, multistate and event-history
models have a long and well-developed history in applied statistics.
Making this connection explicit enables a bidirectional transfer of ideas: algebraic methods
provide new structural insights into multistate models, while classical modelling questions
give rise to new algebraic problems. The use of higher-order Markov chains has increased in recent years, yet they are still not well studied in algebraic statistics. Therefore, we focus on them in this paper. From the multistate model viewpoint, explicit dependence on the history of the process is crucial for interpretation. Expanding the state space to reduce a higher-order Markov chain to a first-order Markov chain makes interpretation more difficult.

One fundamental aspect of multistate and event-history models that we did not address in
detail is the treatment of incomplete data, in particular right censoring.
In many applications, subjects may be lost to follow-up or leave the study before its end, so
that their true state is unknown beyond the censoring time.
One possible approach is to enlarge the state space by introducing censoring as an additional
absorbing state.
This can be natural when censoring is informative, but in the case of independent censoring,
paths involving early censoring may not be of primary interest and may obscure the dynamics
on the original state space.
Understanding how censoring mechanisms can be incorporated into the algebraic framework, and
how they affect the resulting varieties and vanishing ideals, remains an open problem.

\begin{question}
How can right censoring and other forms of incomplete observation be incorporated into the
algebraic-statistical description of multistate Markov models, and how do they affect the
associated algebraic varieties and vanishing ideals?
\end{question}

Our results also clarify the relationship between algebraic and classical likelihood-based
approaches.
Proposition~\ref{prop:mle hierarchical}, together with the computational 
example in
Section~\ref{sec:applications}, shows that for nonhomogeneous first-order multistate Markov
models, algebraic and statistical approaches lead to identical maximum likelihood estimators.
An important direction for future work is to investigate whether this correspondence extends
beyond point estimation.

\begin{question}
How can uncertainty quantification for multistate Markov models, such as Fisher information
or asymptotic variance, be expressed and interpreted within the algebraic framework?
\end{question}

From an algebraic perspective, the most substantial open problems concern homogeneous
multistate Markov models.
While Proposition~\ref{prop:homogeneous binomials} and Remark~\ref{rem:linear} identify
families of binomial relations that necessarily vanish on homogeneous models, explicit
computations indicate that these relations do not generate the full vanishing ideal in
general.

\begin{question}
What is the vanishing ideal of a $k$-th order homogeneous multistate Markov model?
\end{question}

Another natural direction concerns the interaction between homogeneity and structural
constraints such as forbidden transitions or absorbing states.
While in the nonhomogeneous setting such restrictions correspond to coordinate slices of the
general model, their algebraic effect in the homogeneous case is less well understood.

\begin{question}
How do forbidden transitions interact with homogeneity at the level of vanishing ideals?
\end{question}

Finally, several connections to related areas merit further exploration. 
Nonhomogeneous discrete-time multistate Markov models are discrete directed graphical models. We think that homogeneous discrete-time multistate Markov models might have an interpretation as colored discrete graphical models.
Mixtures of 
multistate Markov models are standard tools for modeling population heterogeneity. It would be interesting to study them from an algebraic perspective. 
Connections to reversible Markov chains, models for higher order multistate Markov models, phylogenetic models, and matroid-theoretic structures
suggested by certain binomial relations point to further opportunities for interaction
between algebraic statistics and applied stochastic modelling.

\noindent{\bf Acknowledgement.}
This project was partially funded by the Seed Funding program of KU Leuven, the University of Helsinki, and Aalto University. F.M. was partially supported by the FWO grants G0F5921N (Odysseus), G023721N, and by the KU Leuven grant iBOF/23/064. 

\noindent{\bf Disclosure on the use of 
AI.}
Generative AI tools were used in Section~\ref{sec:applications} for the purposes of quick and efficient data transformation and to assist in smaller coding tasks such as de-bugging or~safety~checks. They were also used throughout the paper to help polish the wording.

\bibliographystyle{plain}
\bibliography{Ref}
\bigskip

{\small{\noindent {\bf Authors' addresses}}

\medskip
\noindent{Dario Gasbarra, University of Vaasa} \hfill {\tt dario.gasbarra@uwasa.fi}
\\
\noindent{Kaie Kubjas, Aalto University}\hfill {\tt kaie.kubjas@aalto.fi}
\\ 
\noindent{Sangita Kulathinal, University of Helsinki} \hfill {\tt sangita.kulathinal@helsinki.fi}
\\
\noindent{Nataliia Kushnerchuk, Aalto University}\hfill {\tt  nataliia.kushnerchuk@aalto.fi} 
\\
\noindent{Fatemeh Mohammadi, 
KU Leuven} \hfill {\tt fatemeh.mohammadi@kuleuven.be}
\\ 
\noindent{Etienne Sebag, University of Helsinki} \hfill{\tt
etienne.sebag@helsinki.fi}}
\newpage\appendix
\section{Path probabilities and ideals for 
Example~\ref{example:revillness_death_model}
}\label{sec:appendix}
{\renewcommand{\arraystretch}{1.2}%
We list the path probabilities for Example~\ref{example:revillness_death_model} of a reversible illness-death model observed at four time instances in Table~\ref{tab:Revobsevations4}. From the path probabilities listed in Table~\ref{tab:Revobsevations4}, we compute the vanishing ideal of the nonhomogeneous model, which has 
66 generators listed below.
\begin{gather*}
    \langle p_{0, 1, 2, 2}p_{1, 1, 1, 2}-p_{0, 1, 1, 2}p_{1, 1, 2, 2},\ p_{1, 0, 1, 2}p_{1, 1, 1, 1}-p_{1, 0, 1, 1}p_{1, 1, 1, 2},\ p_{0, 1, 2, 2}p_{1, 1, 1, 1}-p_{0, 1, 1, 1}p_{1, 1, 2, 2},\\
    p_{0, 1, 1, 2}p_{1, 1,1, 1}-p_{0, 1, 1, 1}p_{1, 1, 1, 2},\ p_{0, 0, 1, 2}p_{1, 1, 1, 1}-p_{0, 0, 1, 1}p_{1, 1, 1, 2},\ p_{1, 0, 1, 2}p_{1, 1, 1, 0}-p_{1, 0, 1, 0}p_{1, 1, 1, 2},\\
    p_{1, 0, 1, 1}p_{1, 1, 1, 0}-p_{1, 0, 1,0}p_{1, 1, 1, 1},\ p_{0, 1, 2, 2}p_{1, 1, 1, 0}-p_{0, 1, 1, 0}p_{1, 1, 2, 2},\ p_{0, 1, 1, 2}p_{1, 1, 1, 0}-p_{0, 1, 1, 0}p_{1, 1, 1, 2},\\
    p_{0, 1, 1, 1}p_{1, 1, 1, 0}-p_{0, 1, 1, 0}p_{1, 1, 1, 1},\ p_{0,0, 1, 2}p_{1, 1, 1, 0}-p_{0, 0, 1, 0}p_{1, 1, 1, 2},\ p_{0, 0, 1, 1}p_{1, 1, 1, 0}-p_{0, 0, 1, 0}p_{1, 1, 1, 1},\\
    p_{0, 1, 2, 2}p_{1, 1, 0, 2}-p_{0, 1, 0, 2}p_{1, 1, 2, 2},\ p_{0, 1, 1, 2}p_{1, 1, 0,2}-p_{0, 1, 0, 2}p_{1, 1, 1, 2},\ p_{0, 1, 1, 1}p_{1, 1, 0, 2}-p_{0, 1, 0, 2}p_{1, 1, 1, 1},\\
    p_{0, 1, 1, 0}p_{1, 1, 0, 2}-p_{0, 1, 0, 2}p_{1, 1, 1, 0},\ p_{1, 0, 0, 2}p_{1, 1, 0, 1}-p_{1, 0, 0, 1}p_{1,1, 0, 2},\ p_{0, 1, 2, 2}p_{1, 1, 0, 1}-p_{0, 1, 0, 1}p_{1, 1, 2, 2},\\
    p_{0, 1, 1, 2}p_{1, 1, 0, 1}-p_{0, 1, 0, 1}p_{1, 1, 1, 2},\ p_{0, 1, 1, 1}p_{1, 1, 0, 1}-p_{0, 1, 0, 1}p_{1, 1, 1, 1},\ p_{0, 1, 1,0}p_{1, 1, 0, 1}-p_{0, 1, 0, 1}p_{1, 1, 1, 0},\\
    p_{0, 1, 0, 2}p_{1, 1, 0, 1}-p_{0, 1, 0, 1}p_{1, 1, 0, 2},\ p_{0, 0, 0, 2}p_{1, 1, 0, 1}-p_{0, 0, 0, 1}p_{1, 1, 0, 2},\ p_{1, 0, 0, 2}p_{1, 1, 0, 0}-p_{1,0, 0, 0}p_{1, 1, 0, 2},\\
    p_{1, 0, 0, 1}p_{1, 1, 0, 0}-p_{1, 0, 0, 0}p_{1, 1, 0, 1},\ p_{0, 1, 2, 2}p_{1, 1, 0, 0}-p_{0, 1, 0, 0}p_{1, 1, 2, 2},\ p_{0, 1, 1, 2}p_{1, 1, 0, 0}-p_{0, 1, 0, 0}p_{1, 1, 1, 2},\\
     p_{0, 1, 1, 1}p_{1, 1, 0, 0}-p_{0, 1, 0, 0}p_{1, 1, 1, 1},\ p_{0, 1, 1, 0}p_{1, 1, 0, 0}-p_{0, 1, 0, 0}p_{1, 1, 1, 0},\ p_{0, 1, 0, 2}p_{1, 1, 0, 0}-p_{0, 1, 0, 0}p_{1, 1, 0, 2},\\
     p_{0, 1, 0, 1}p_{1, 1,0, 0}-p_{0, 1, 0, 0}p_{1, 1, 0, 1},\ p_{0, 0, 0, 2}p_{1, 1, 0, 0}-p_{0, 0, 0, 0}p_{1, 1, 0, 2},\ p_{0, 0, 0, 1}p_{1, 1, 0, 0}-p_{0, 0, 0, 0}p_{1, 1, 0, 1},\\
     p_{0, 0, 2, 2}p_{1, 0, 1, 2}-p_{0, 0, 1,2}p_{1, 0, 2, 2},\ p_{0, 1, 1, 2}p_{1, 0, 1, 1}-p_{0, 1, 1, 1}p_{1, 0, 1, 2},\ p_{0, 0, 2, 2}p_{1, 0, 1, 1}-p_{0, 0, 1, 1}p_{1, 0, 2, 2},\\
     p_{0, 0, 1, 2}p_{1, 0, 1, 1}-p_{0, 0, 1, 1}p_{1, 0, 1, 2},\ p_{0,1, 1, 2}p_{1, 0, 1, 0}-p_{0, 1, 1, 0}p_{1, 0, 1, 2},\ p_{0, 1, 1, 1}p_{1, 0, 1, 0}-p_{0, 1, 1, 0}p_{1, 0, 1, 1},\\
     p_{0, 0, 2, 2}p_{1, 0, 1, 0}-p_{0, 0, 1, 0}p_{1, 0, 2, 2},\ p_{0, 0, 1, 2}p_{1, 0, 1,0}-p_{0, 0, 1, 0}p_{1, 0, 1, 2},\ p_{0, 0, 1, 1}p_{1, 0, 1, 0}-p_{0, 0, 1, 0}p_{1, 0, 1, 1},\\
     p_{0, 0, 2, 2}p_{1, 0, 0, 2}-p_{0, 0, 0, 2}p_{1, 0, 2, 2},\ p_{0, 0, 1, 2}p_{1, 0, 0, 2}-p_{0, 0, 0, 2}p_{1,0, 1, 2},\ p_{0, 0, 1, 1}p_{1, 0, 0, 2}-p_{0, 0, 0, 2}p_{1, 0, 1, 1},\\
     p_{0, 0, 1, 0}p_{1, 0, 0, 2}-p_{0, 0, 0, 2}p_{1, 0, 1, 0},\ p_{0, 1, 0, 2}p_{1, 0, 0, 1}-p_{0, 1, 0, 1}p_{1, 0, 0, 2},\ p_{0, 0, 2,2}p_{1, 0, 0, 1}-p_{0, 0, 0, 1}p_{1, 0, 2, 2},\\
     p_{0, 0, 1, 2}p_{1, 0, 0, 1}-p_{0, 0, 0, 1}p_{1, 0, 1, 2},\ p_{0, 0, 1, 1}p_{1, 0, 0, 1}-p_{0, 0, 0, 1}p_{1, 0, 1, 1},\ p_{0, 0, 1, 0}p_{1, 0, 0, 1}-p_{0, 0, 0, 1}p_{1, 0, 1, 0},\\
     p_{0, 0, 0, 2}p_{1, 0, 0, 1}-p_{0, 0, 0, 1}p_{1, 0, 0, 2},\ p_{0, 1, 0, 2}p_{1, 0, 0, 0}-p_{0, 1, 0, 0}p_{1, 0, 0, 2},\ p_{0, 1, 0, 1}p_{1, 0, 0, 0}-p_{0, 1, 0, 0}p_{1, 0, 0, 1},\\
      p_{0, 0, 2, 2}p_{1, 0, 0, 0}-p_{0, 0, 0, 0}p_{1, 0, 2, 2},\ p_{0, 0, 1, 2}p_{1, 0, 0, 0}-p_{0, 0, 0, 0}p_{1, 0, 1, 2},\ p_{0, 0, 1, 1}p_{1, 0, 0, 0}-p_{0, 0, 0, 0}p_{1, 0, 1, 1},\\
      p_{0, 0, 1, 0}p_{1, 0,0, 0}-p_{0, 0, 0, 0}p_{1, 0, 1, 0},\ p_{0, 0, 0, 2}p_{1, 0, 0, 0}-p_{0, 0, 0, 0}p_{1, 0, 0, 2},\ p_{0, 0, 0, 1}p_{1, 0, 0, 0}-p_{0, 0, 0, 0}p_{1, 0, 0, 1},\\
      p_{0, 0, 1, 2}p_{0, 1, 1, 1}-p_{0, 0, 1,1}p_{0, 1, 1, 2},\ p_{0, 0, 1, 2}p_{0, 1, 1, 0}-p_{0, 0, 1, 0}p_{0, 1, 1, 2},\ p_{0, 0, 1, 1}p_{0, 1, 1, 0}-p_{0, 0, 1, 0}p_{0, 1, 1, 1},\\
      p_{0, 0, 0, 2}p_{0, 1, 0, 1}-p_{0, 0, 0, 1}p_{0, 1, 0, 2},\ p_{0,0, 0, 2}p_{0, 1, 0, 0}-p_{0, 0, 0, 0}p_{0, 1, 0, 2},\ p_{0, 0, 0, 1}p_{0, 1, 0, 0}-p_{0, 0, 0, 0}p_{0, 1, 0, 1} \rangle.
\end{gather*}

\begin{table}[!ht]
    \centering 
  \caption{Potential observations at four time points under the reversible illness-death model (Figure~\ref{fig:illness} with transition from 1 to 0 included). The initial state can be $0$ or $1$. State $2$ is an absorbing state and hence, the chain will stay in state $2$ once that state is reached. The transition probability from state $2$ to states $0$ and $1$ is zero and to state $2$ is one. Hence, all probabilities of the type $a_{20}^{(\ell)} = 0, a_{21}^{(\ell)} = 0, a_{22}^{(\ell)} = 1,$ for all $\ell$.}
\label{tab:Revobsevations4}
    \begin{tabular}{c|c|c|r}
    \hline
        Initial & Three next instances & Prob (nonhomogeneous) & Prob (homogeneous)\\
        \hline
        0 & 000 & $\pi_0a_{00}^{(1)} a_{00}^{(2)} a_{00}^{(3)}$ & $\pi_0a_{00}^3$ \\
        0 & 001  & $\pi_0a_{00}^{(1)} a_{00}^{(2)} a_{01}^{(3)}$ & $\pi_0a_{00}^2 a_{01}$ \\
        0 & 002  & $\pi_0a_{00}^{(1)} a_{00}^{(2)} a_{02}^{(3)}$ & $\pi_0a_{00}^2 a_{02}$ \\
        0 & 011  & $\pi_0a_{00}^{(1)} a_{01}^{(2)} a_{11}^{(3)}$ & $\pi_0a_{00} a_{01} a_{11}$ \\
        0 & 012  & $\pi_0a_{00}^{(1)} a_{01}^{(2)} a_{12}^{(3)}$ & $\pi_0a_{00} a_{01} a_{12}$ \\
        0 & 022  & $\pi_0a_{00}^{(1)} a_{02}^{(2)} a_{22}^{(3)}$ & $\pi_0a_{00} a_{02} a_{22}$ \\
        0 & 111  & $\pi_0a_{01}^{(1)} a_{11}^{(2)} a_{11}^{(3)}$ & $\pi_0a_{01} a_{11}^2$ \\
        0 & 112  & $\pi_0a_{01}^{(1)} a_{11}^{(2)} a_{12}^{(3)}$ & $\pi_0a_{01} a_{11} a_{12}$ \\
        0 & 122  & $\pi_0a_{01}^{(1)} a_{12}^{(2)} a_{22}^{(3)}$ & $\pi_0a_{01} a_{12} a_{22}$\\
        0 & 222  & $\pi_0a_{02}^{(1)} a_{22}^{(2)} a_{22}^{(3)}$ & $\pi_0a_{02} a_{22}^2$\\
        0 & 100  & $\pi_0a_{01}^{(1)} a_{00}^{(2)} a_{00}^{(3)}$ & $\pi_0a_{00}^2 a_{01}$ \\
        0 & 010  & $\pi_0a_{00}^{(1)} a_{01}^{(2)} a_{10}^{(3)}$ & $\pi_0a_{00} a_{01} a_{10}$ \\
        0 & 110  & $\pi_0a_{01}^{(1)} a_{11}^{(2)} a_{10}^{(3)}$ & $\pi_0a_{01} a_{10} a_{11}$ \\        
        0 & 101  & $\pi_0a_{01}^{(1)} a_{10}^{(2)} a_{01}^{(3)}$ & $\pi_0a_{01}^2 a_{10}$ \\ 
        0 & 102 & $\pi_0a_{01}^{(1)} a_{10}^{(2)} a_{02}^{(3)}$ & $\pi_0a_{01}a_{10}a_{02}$ \\  
        \hline
        1 & 000 & $\pi_1a_{10}^{(1)} a_{00}^{(2)} a_{00}^{(3)}$ & $\pi_1a_{10} a_{00}^2$ \\
        1 & 001  & $\pi_1a_{10}^{(1)} a_{00}^{(2)} a_{01}^{(3)}$ & $\pi_1a_{10} a_{00} a_{01}$ \\
        1 & 002  & $\pi_1a_{10}^{(1)} a_{00}^{(2)} a_{02}^{(3)}$ & $\pi_1a_{10} a_{00} a_{02}$ \\
        1 & 011  & $\pi_1a_{10}^{(1)} a_{01}^{(2)} a_{11}^{(3)}$ & $\pi_1a_{10} a_{01} a_{11}$ \\
        1 & 012  & $\pi_1a_{10}^{(1)} a_{01}^{(2)} a_{12}^{(3)}$ & $\pi_1a_{10} a_{01} a_{12}$ \\
        1 & 022  & $\pi_1a_{10}^{(1)} a_{02}^{(2)} a_{22}^{(3)}$ & $\pi_1a_{10} a_{02} a_{22}$ \\
        1 & 111  & $\pi_1a_{11}^{(1)} a_{11}^{(2)} a_{11}^{(3)}$ & $\pi_1a_{11}^3$ \\
        1 & 112  & $\pi_1a_{11}^{(1)} a_{11}^{(2)} a_{12}^{(3)}$ & $\pi_1a_{11}^2 a_{12}$ \\
        1 & 122  & $\pi_1a_{11}^{(1)} a_{12}^{(2)} a_{22}^{(3)}$ & $\pi_1a_{11} a_{12} a_{22}$\\
        1 & 222  & $\pi_1a_{12}^{(1)} a_{22}^{(2)} a_{22}^{(3)}$ & $\pi_1a_{12} a_{22}^2$\\
        1 & 100  & $\pi_1a_{11}^{(1)} a_{00}^{(2)} a_{00}^{(3)}$ & $\pi_1a_{11} a_{00}^2$ \\
        1 & 010  & $\pi_1a_{10}^{(1)} a_{01}^{(2)} a_{10}^{(3)}$ & $\pi_1a_{01} a_{10}^2$ \\
        1 & 110  & $\pi_1a_{11}^{(1)} a_{11}^{(2)} a_{10}^{(3)}$ & $\pi_1a_{10} a_{11}^2$ \\        
        1 & 101  & $\pi_1a_{11}^{(1)} a_{10}^{(2)} a_{01}^{(3)}$ & $\pi_1a_{01} a_{10} a_{11}$ \\           1 & 102 & $\pi_1a_{01}^{(1)} a_{10}^{(2)} a_{02}^{(3)}$ & $\pi_1a_{01}a_{10}a_{02}$ \\     
        \hline        
    \end{tabular}
\end{table}
}

\clearpage

\end{document}